\documentclass
{article}
\usepackage{epsfig,latexsym,xspace,oupbib,url,subcaption}
\usepackage{a4,amsopn,
latexsym,
xspace,graphicx,
hyperref,
mathrsfs,epsfig,bm,bbm,enumitem,geometry,marginnote}
\newcommand{\CIP}{\mathop{\perp\!\!\!\perp}}

\title{Decision-theoretic foundations for statistical causality}
\author{A.~Philip~Dawid\thanks{University of Cambridge}}

\newcommand{\ie}{{\em i.e.\/}\xspace}
\newcommand{\eg}{{\em e.g.\/}\xspace}

\newcommand{\etc}{{\em etc.\/}\xspace}

\newcommand{\HIDE}[1]{ }
\newcommand{\COMMENT}[1]{ }

\renewcommand{\dag}{{\cal D}}

\newcommand{\half}{\frac{1}{2}}

\newcommand{\eqref}[1]{\mbox{(\ref{eq:#1})}}

\newcommand{\appref}[1]{\mbox{Appendix~\ref{sec:#1}}}
\newcommand{\secref}[1]{\mbox{\S$\,$\ref{sec:#1}}}

\newcommand{\figref}[1]{\mbox{Figure~\ref{fig:#1}}}
\newcommand{\lemref}[1]{\mbox{Lemma~\ref{lem:#1}}}

\newcommand{\tabref}[1]{\mbox{Table~\ref{tab:#1}}}
\newcommand{\defref}[1]{\mbox{Definition~\ref{def:#1}}}
\newcommand{\fnref}[1]{\mbox{footnote~\ref{fn:#1}}}

\newcommand{\itref}[1]{\mbox{\ref{it:#1}}}

\newcommand{\thmref}[1]{\mbox{Theorem~\ref{thm:#1}}}

\newcommand{\corref}[1]{\mbox{Corollary~\ref{cor:#1}}}
\newcommand{\condref}[1]{\mbox{Condition~\ref{cond:#1}}}
\newcommand{\remref}[1]{\mbox{Remark~\ref{rem:#1}}}

\newcommand{\Secref}[1]{\mbox{Section~\ref{sec:#1}}}

\newcommand{\pre}{{\rm pre}}
\newcommand{\suc}{{\rm suc}}

\newcommand{\cd}{\,|\,}

\newcommand{\E}{{\mbox{E}}}

\newcommand{\var}{{\mbox{var}}}

\newcommand{\cip}{\mbox{$\perp\!\!\!\perp$}}
\newcommand{\indo}[2]{\mbox{$#1 \,\cip\, #2$}}
\newcommand{\ind}[3]{\mbox{$#1 \, \cip\, #2 \mid #3$}}

\newtheorem{expl}{Example}
\newtheorem{definer}{Definition}
\newtheorem{theorem}{Theorem}

\newtheorem{algor}{Algorithm}
\newtheorem{lemma}{Lemma}
\newtheorem{lem}[lemma]{Lemma}

\newtheorem{cor}{Corollary}
\newtheorem{cond}{Condition}
\newtheorem{rem*}{Remark}

\newcommand{\halm}{\hspace*{\fill} $\Box$\par}
\newenvironment{proof}{\noindent {\bf Proof. }}{\halm\vspace{\baselineskip}}
\newenvironment{proof0}[1]{\noindent {\bf Proof of
    {#1}. }}{\halm\vspace{\baselineskip}}
\newenvironment{ex}{\begin{expl}\rm}{\halm\end{expl}}
\newenvironment{rem}{\begin{rem*}\rm}{\halm\end{rem*}}
\newenvironment{defn}{\begin{definer}\rm}{\halm\end{definer}}

\newcommand{\pr}[1]{{p}(#1)}

\renewcommand{\condref}[1]{\mbox{Condition~\ref{cond:#1}}}
 \renewcommand{\theenumi}{(\roman{enumi})}
\renewcommand{\pr}{\mbox{\rm pr}}

\newcommand{\bmu}{\mbox{\boldmath$\mu$}}

\newcommand{\norm}{\mbox{${\cal N}$}}
\newcommand{\idle}{\mbox{$\emptyset$}}

\newcommand{\ice}{\mbox{\rm ICE}\xspace}
\newcommand{\ace}{\mbox{\rm ACE}\xspace}

\newcommand{\citep}{\cite}
\newcommand{\citet}{\textcite}

\begin{document}
\maketitle

\begin{abstract}
  \noindent We develop a mathematical and interpretative foundation
  for the enterprise of decision-theoretic statistical causality (DT),
  which is a straightforward way of representing and addressing causal
  questions.  DT reframes causal inference as ``assisted
  decision-making'', and aims to understand when, and how, I can make
  use of external data, typically observational, to help me solve a
  decision problem by taking advantage of assumed relationships
  between the data and my problem.

  The relationships embodied in any representation of a causal problem
  require deeper justification, which is necessarily
  context-dependent.  Here we clarify the considerations needed to
  support applications of the DT methodology.  Exchangeability
  considerations are used to structure the required relationships, and
  a distinction drawn between intention to treat and intervention to
  treat forms the basis for the enabling condition of
  ``ignorability''.
  
  We also show how the DT perspective unifies and sheds light on other
  popular formalisations of statistical \mbox{causality}, including
  potential responses and directed acyclic graphs.
  \\

\end{abstract}

\section{Introduction}
\label{sec:intro}

The decision-theoretic (DT) approach to statistical causality has been
described and developed in a series of papers
\cite{apd:cinfer,apd:infdiags,apd:hsss,vd/apd/sg:uai,apd:kent,sara:07,apd/vd:uai08,dd:ss,hg/apd:aistats2010,sgg/apd:ett,apd:causDTchapter,carlo/phil/vanessa:dynamicchapter,apd/pc:2014,hg/apd/gmb:suffcov};
for general overview see \textcite{apd:gtp,apd:annrev}.  It has been
shown to be a more straightforward approach, both philosophically and
for use in applications, than other popular frameworks for statistical
causality based \eg\ on potential responses or directed acyclic
graphs.

From the standpoint of DT, ``causal inference'' is something of a
misnomer for the great preponderance of the methodological and applied
contributions that normally go by this description.  A better
characterisation of the field would be ``assisted decision making''.
Thus the DT approach focuses on how we might make use of
external---typically observational---data to help inform a
decision-maker how best to act; it aims to characterise conditions
allowing this, and to develop ways in which it can be achieved.  Work
to date has concentrated on the nuts and bolts of showing how the DT
approach may be applied to a variety of problems, but has largely
avoided any detailed consideration of how the conditions enabling such
application might be justified in terms of still more fundamental
assumptions.  The main purpose of the present paper is to to conduct a
careful and rigorous analysis, to serve as a foundational ``prequel''
to the DT enterprise.  We develop, in detail, the basic structures and
assumptions that, when appropriate, would justify the use of a DT
model in a given context---a step largely taken for granted in earlier
work.  We emphasise important distinctions, such as that between cause
and effect variables, and that between intended and applied treatment,
both of which are reflected in the formal language; another important
distinction is that between post-treatment and pre-treatment
exchangeability.  The rigorous development is based on the algebraic
theory of extended conditional independence, which admits both
stochastic and non-stochastic variables
\cite{apd:CIST,apd:ciso,pc/apd:eci}, and its graphical representation
\cite{apd:infdiags}.

We also consider the relationships between DT and alternative current
formulations of statistical causality, including potential outcomes
\cite{dbr:jep,dbr:as}, Pearlian DAGs \cite{pearl:book}, and single
world intervention graphs
\cite{Richardson_primer,Richardson_singleworld}.  We develop DT
analogues of concepts that have been considered fundamental in these
alternative approaches, including consistency, ignorability, and the
stable unit-treatment value assumption.  In view of these connexions,
we hope that this foundational analysis of DT causality will also be
of interest and value to those who would seek a deeper understanding
of their own preferred causal framework, and in particular of the
conditions that need to be satisfied to justify their models.

\subsection*{Plan of paper}
\Secref{dt} describes, with simple examples, the basics of the DT
approach to modelling problems of ``statistical causality'', noting in
particular the usefulness of introducing a non-stochastic variable
that allows us to distinguish between the different
regimes---observational and interventional---of interest.  It shows
how assumed relationships between these regimes, intended to support
causal inference, may be fruitfully expressed using the language and
notation of extended conditional independence, and represented
graphically by means of an augmented directed acyclic graph.

In \secref{agency} and \secref{simple} we describe and illustrate the
standard approach to modelling a decision problem, as represented by a
decision tree.  The distinction between cause and effect is reflected
by regarding a cause as a non-stochastic decision variable, under the
external control of the decision-maker, while an effect is a
stochastic variable, that can not be directly controlled in this way.
We introduce the concept of the hypothetical distribution for an
effect variable, were a certain action to be taken, and point out that
all we need, to solve the decision problem, is the collection of all
such hypothetical distributions.

\Secref{populating} frames the purpose of ``causal inference'' as
taking advantage of external data to help me solve my decision
problem, by allowing me to update my hypothetical distributions
appropriately.  This is elaborated in \secref{exchpred}, where we
relate the external data to my own problem by means of the concept of
exchangeability.  We distinguish between post-treatment
exchangeability, which allows straightforward use of the data, and
pre-treatment exchangeability, which can not so use the data without
making further assumptions.  These assumptions---especially,
ignorability---are developed in \secref{assig}, in terms of a clear
formal distinction between intention to treat and intervention to
treat.  In \secref{idle} we develop this formalism further,
introducing the non-stochastic regime indicator that is central to the
DT formulation.  \Secref{cov} generalises this by introducing
additional covariate information, while \secref{complex} generalises
still further to problems represented by a directed acyclic graph.  In
\secref{comparison} we highlight similarities and differences between
the DT approach to statistical causality and other formalisms,
including potential outcomes, Pearlian DAGs, and single-world
intervention graphs.  These comparisons and contrasts are explored
further in \secref{gcomp}, by application to a specific problem, and
it is shown how the DT approach brings harmony to the babel of
different voices.  \Secref{disc} rounds off with a general discussion
and suggestions for further developments.  Some technical proofs are
relegated to \appref{thmproof}.

\section{The DT approach}
\label{sec:dt}

Here we give a brief overview of the DT perspective on modelling
problems of statistical causality.

A fundamental feature of the DT approach is its consideration of the
relationships between the various probability distributions that
govern different regimes of interest.  As a very simple example,
suppose that we have a binary treatment variable $T$, and a response
variable $Y$.  We consider three different regimes, indexed by the
values of a non-stochastic regime indicator variable
$F_T$:\footnote{The use of explicit intervention variables such as
  $F_T$ was pioneered by \textcite{pearl:isi,pearl:sdlc.disc},
  although, for reasons obscure to this author, he seems largely to
  have abandoned it very quickly.}
\begin{description}
\item[$F_T = 1$.]  This is the regime in which the active treatment is
  administered to the patient
\item [$F_T = 0$.] This is the regime in which the control treatment is
  administered to the patient
\item[$F_T = \idle$.] This is a regime in which the choice of
  treatment is left to some uncontrolled external source.
\end{description}

The first two regimes may be described as {\em interventional\/}, and
the last as {\em observational\/}.  In each regime there will be a
joint distribution for the treatment and response variables, $T$ and
$Y$.  The distribution of $T$ will be degenerate under an
interventional regime (with $T=1$ almost surely under $F_T = 1$ and
$T=0$ almost surely under $F_T = 0$); but $T$ will typically have a
non-degenerate distribution in the observational regime

It will often be the case that I have access to data collected under
the observational regime $F_T = \idle$; but for decision-making
purposes I am interested in comparing and choosing between two
interventions available to me, $F_T = 1$ and $F_T = 0$, for which I do
not have direct access to relevant data.  I can only use the
observational data to address my decision problem if I can make, and
justify, appropriate assumptions relating the distributions associated
with the different regimes.

The simplest such assumption (which, however, will often not be easy
to justify) is that the distribution of $Y$ in the interventional
active treatment regime $F_T=1$ is the same as the conditional
distribution of $Y$, given $T=1$, in the observational regime
$F_T=\idle$; and likewise the distribution of $Y$ under regime $F_T=0$
is the same as the conditional distribution of $Y$ given $T=0$ in the
regime $F_T=\idle$.  This assumption can be expressed, in the
conditional independence notation of \textcite{apd:CIST}, as:
\begin{equation}
\label{eq:simple}
  \ind Y {F_T} T,
\end{equation}
(read: ``$Y$ is independent of $F_T$, given $T$''), which asserts that
the conditional distributions of the response $Y$, given the
administered treatment $T$, does not further depend on $F_T$ (\ie, on
whether that treatment arose naturally, in the observational regime,
or by an imposed intervention), and so can be chosen to be the same in
all three regimes.

Note, importantly, that the conditional independence assertion
\eqref{simple} makes perfect intuitive sense, even though the variable
$F_T$ that occurs in it is non-stochastic.
The intuitive content of \eqref{simple} is made fully rigorous by the
theory of extended conditional independence (ECI)
\cite{apd:ciso,pc/apd:eci}, which shows that such expressions can,
with care, be manipulated in exactly the same way as when all
variables are stochastic.

Property~\eqref{simple} can also be expressed graphically, by the
augmented DAG (directed acyclic graph) \cite{apd:infdiags} of
\figref{simple}.  Again, we can include both stochastic variables
(represented by round nodes) and non-stochastic variables (square
nodes) in such a graph, which encodes extended conditional
independence by means of the $d$-separation criterion
\cite{geiger/verma/pearl:90} or the equivalent moralisation criterion
\cite{sll/apd/bnl/hgl:directed}.  In \figref{simple} it is the absence
of an arrow from $F_T$ to $Y$ that encodes property~\eqref{simple}.

\begin{figure}[htbp]
  \begin{center}
    \resizebox{1.5in}{!}{\includegraphics{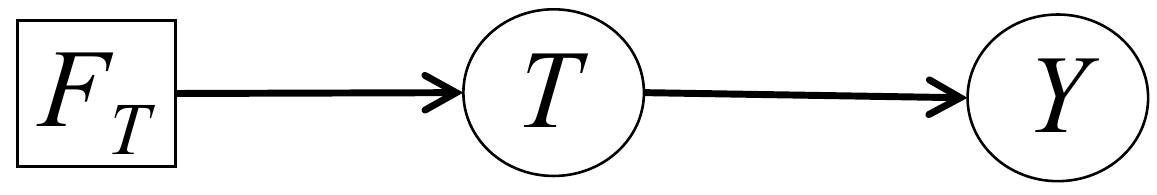}}
    \caption{A simple augmented DAG}
    \label{fig:simple}
  \end{center}
\end{figure}

The identity, expressed by \eqref{simple}, of the conditional
distribution of $Y$ given $T$, across all the regimes described by the
values of the regime indicator $F_T$, can be understood as expressing
the {\em invariance\/} or {\em stability\/}
\cite{bhlmann2018invariance} of a probabilistic ingredient---the
conditional distribution of $Y$, given $T$---across the different
regimes.  This is thus being regarded as a modular component,
unchanged wherever it appears in any of the regimes.  When it can be
justified, the stability property represented by \eqref{simple} or
\figref{simple} permits {\em transfer\/} \cite{pearl/bareinboim} of
relevant information between the regimes: we can use the (available,
but not directly interesting) observational data to estimate the
distributions of response $Y$ given treatment $T$ in regime
$F_T=\idle$; and then regard these observational conditional
distributions as also supplying the desired interventional
distributions of $Y$ (of interest, but not directly available) in the
hypothetical regimes $F_T=1$ and $F_T=0$ relevant to my decision
problem.\footnote{An important aside on notation and terminology.  In
  the potential outcome (PO) approach, the response $Y$ is
  artificially split into two, $Y_0$ and $Y_1$, it being supposed that
  $Y_t$ is what is observed in regime $F_t$---the marginal
  distribution of $Y_t$ thus being the same as our hypothetical
  distribution for $Y$ under intervention $F_T=t$.  This duplication
  of the response is entirely unnecessary for our purposes.  Moreover,
  there is a very prevalent misuse of terms such as ``counterfactual
  distribution'', or ``estimating the counterfactual'',
  notwithstanding that there is nothing counter to any known fact
  involved in considering these distributions, which are to be applied
  to a new case.  We have termed the interventional distributions of
  $Y$ {\em hypothetical\/}, since they are predicated on a
  hypothetical intervention on a new case.  I have elsewhere
  \cite{apd:kent} expanded on the importance of distinguishing between
  hypothetical and counterfactual reasoning, which is jeopardised when
  we do not also make a clear terminological distinction.}
Characterising, justifying, and capitalising on such modularity
properties are core features of the DT approach to causality.

A more complex example is given by the DAG of \figref{instrument},
which represents a problem where $Z$ is an instrumental variable for
the effect of a binary exposure variable $X$ on an outcome variable
$Y$, in the presence of unobserved ``confounding variables'' $U$.
Note again the inclusion of the regime indicator $F_X$, with values
$0$, $1$ and $\idle$.  As before, $F_X = \idle$ labels the
observational regime in which data are actually obtained, while
$F_X = 1$ [resp., 0] labels the hypothetical regime where we intervene
to force $X$ to take the value 1 [resp., 0].
 
\begin{figure}[htbp]
  \begin{center}
    \resizebox{1.5in}{!}{\includegraphics{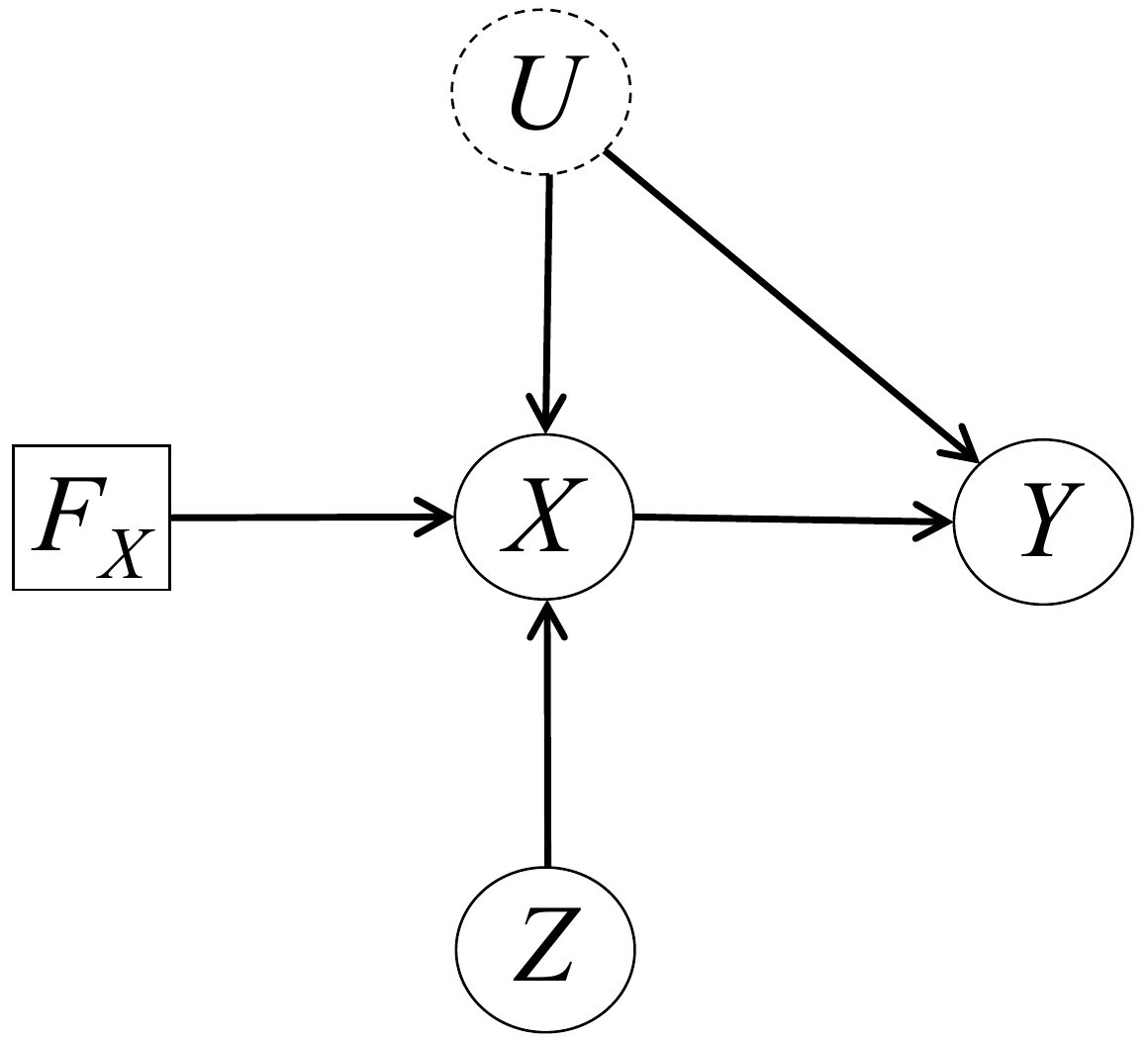}}
    \caption{Instrumental variable with regimes}
    \label{fig:instrument}
  \end{center}
\end{figure}

The figure is nothing more nor less than the graphical representation
of the following extended conditional independence properties (which
it embodies by means of $d$-separation):
\begin{eqnarray}
  \label{eq:instr0}
  (Z,U)   &\cip& F_X\\
  \label{eq:instr1}
  U &\cip& Z  \,\,\,\,\,\,\cd \,\, F_X\\
  \label{eq:instr2}
   Y &\cip& Z \,\,\,\,\,\,\cd \,\,(X, U, F_X)\\
  \label{eq:instr3}
  Y &\cip& F_X \,\,\cd \,\,(X, U).
\end{eqnarray}
In words, \eqref{instr0} asserts that the joint distribution of $Z$
and $U$ is a modular component, the same in all 3 regimes, while
\eqref{instr1} further requires that, in this (common) joint
distribution, we have independence between $U$ and $Z$.  Next,
\eqref{instr2} says that, in any regime, the response $Y$ is
independent of the instrument $Z$, conditionally on exposure $X$ and
confounders $U$ (the ``exclusion restriction''); while \eqref{instr3}
further requires that the conditional distribution for $Y$, given $X$
and $U$ (which, by \eqref{instr2}, is unaffected by further
conditioning on $Z$) be the same in all regimes.

We emphasise that properties \eqref{instr0}--\eqref{instr3} comprise
the full extent of the causal assumptions made.  In particular---and
in contrast to other common interpretations of a ``causal graph''
\cite{apd:beware}--- no further causal conclusions should be drawn
from the directions of the arrows in \figref{instrument}.  In
particular, the arrow from $Z$ to $X$ should not be interpreted as
implying a causal effect of $Z$ on $X$: indeed, the figure is fully
consistent with alternative causal assumptions, for example that $Z$
and $X$ are merely associated by sharing a common cause
\cite{apd:beware}.  In general, the causal content of any augmented
DAG is to be understood as fully comprised by the extended conditional
independencies that it embodies by $d$-separation.  This gives a
precise and clear semantics to our ``causal DAGs''.

To the extent that the assumptions embodied in \figref{instrument}
imply restrictions on the observational distribution of the data (\ie,
properties \eqref{instr1} and \eqref{instr2}, considered only under
the operation of the observational regime $F_X = \idle$), they tally
with the standard assumptions made in instrumental variable analysis
\cite{hernan:06}.  However, without the additional stitching together
of behaviours under the observational regime and the desired, but
unobserved, interventional regimes, it is not possible to use the
observational data to make causal inferences.  When, and only when,
these additional stability assumptions can be made can we justify
application of the usual methods of instrumental variable analysis.

In previous work, we have used the above formulation in terms of
extended conditional independences, involving both stochastic
variables and non-stochastic regime indicators, as the starting point
for analysis and discussion of statistical causality, both in general
terms and in particular applications.  In this work, we aim to dig a
little deeper into the foundations, and in particular to understand
why, when, and how we might justify the specific extended conditional
independence properties previously simply assumed.

\section{Causality, agency and decision}
\label{sec:agency}

There is a very wide variety of philosophical understandings and
interpretations of the concept of ``causality''.  Our own approach is
closely aligned with the ``agency'' interpretation
\cite{reichenbach,price:bjps91,hausman:book,woodward:book,woodward:sep},
whereby a ``cause'' is understood as something that can (at least in
principle) be externally manipulated---this notion being an undefined
primitive, whose intended meaning is easy enough to comprehend
intuitively in spite of being philosophically contentious \cite{webb:NS}.
This is not to deny the value of other interpretations of causality,
based for example on mechanisms \cite{salmon:book,dowe:book},
simplicity \cite{algmark}, probabilistic independence
\cite{suppes:book,spohn:mcg} or invariant processes
\cite{bhlmann2018invariance}, or starting from different primitive
notions, such as common cause or direct effect \cite{sgs:book}, or one
variable ``listening to'' another \cite{pearl:why}.  However, the
present work has the limited aim of explicating the agency-based
decision-theoretic approach.

The basic idea is that an agent (``I'', say) has free choice among a
set of available actions, and that performing an action will, in some
sense, tend to bring about some outcome.  Indeed, whenever I seriously
contemplate performing some action, my purpose is to bring about some
desired outcome; and that aim will inform my choice between the
different actions that may be available.  We may consider my action as
a putative ``cause'' of my outcome.  This approach makes a clear
distinction between cause and effect: the former is represented as an
action, subject to my free choice, while the latter is represented as
an outcome variable, over which I have no direct control.
Correspondingly, we will need different formal representations for
cause and effect variables: only the latter will be treated as
stochastic random variables.

Now by my action I generally won't be able to determine the outcome
exactly, since it will also be affected by many circumstances beyond
my control, which we might ascribe to the vagaries of ``Nature''.  So
I will have uncertainty about the eventual outcome that would ensue
from my action.  We shall take it for granted that it is always
appropriate to represent my uncertainty by a probability distribution.
Then, for any contemplated but not yet executed action $a$, there will
be a joint probability distribution $P_a$ over all the ensuing
variables in the problem%
\footnote{In full generality, the relevant collection of ensuing
  variables could itself depend on my action $a$; purely for
  simplicity we shall restrict to the case that it does not.},
representing my current uncertainty (conditioned on whatever knowledge
I currently have, prior to choosing my action) about how those
variables {\em might\/} turn out, {\em were I to perform\/} action
$a$.  We will term such a distribution $P_a$ {\em hypothetical\/},
since it is premised on the {\em hypothesis\/} that I perform action
$a$.

There will be a collection ${\cal A}$ of actions available to me, and
correspondingly an associated collection $\{P_a: a\in{\cal A}\}$ of my
hypothetical distributions---each contingent on just one of the
actions I might take.  My task is to rank my preferences among these
different hypothetical distributions over future outcomes, and perform
that action corresponding to the distribution $P_a$ I like best.  I
can do this ranking in terms of any feature of the distributions that
interests me.

One such way, concordant with Bayesian statistical decision theory
\cite{raiffa:schlaifer:61,degroot:70}, is to construct a real-valued
loss function $L$, such that $L(y,a)$ measures the dissatisfaction I
will suffer if I take action $a$ and the value of some associated
outcome variable $Y$ later turns out to be $y$. This is represented in
the decision tree of \figref{0}.
\begin{figure}[htbp]
\begin{center}
  \resizebox{3in}{!}{\includegraphics{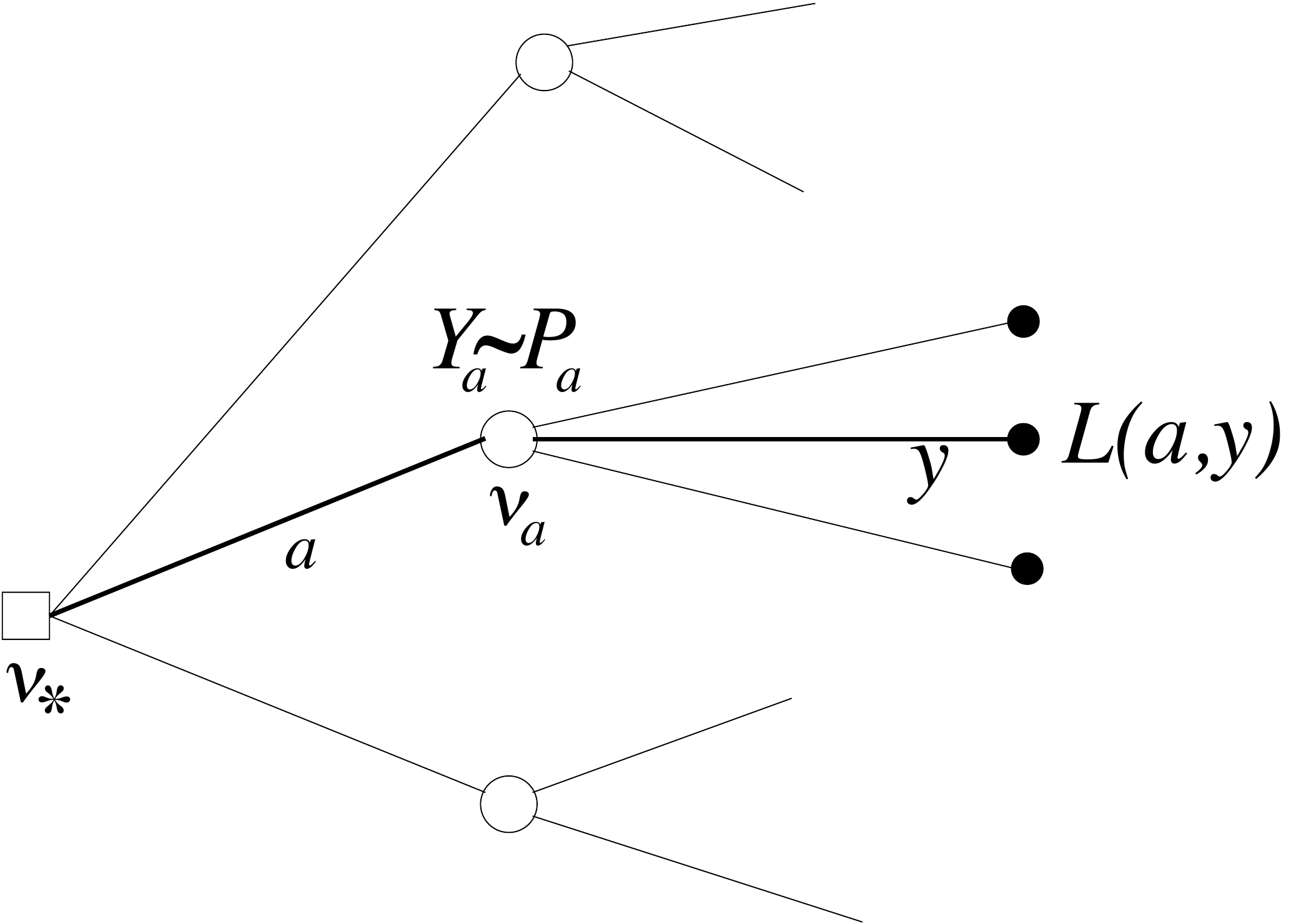}}
  \caption{Decision tree}
\label{fig:0}
\end{center}
\end{figure}


The square at node $\nu_*$ indicates that it is a decision node, where
I can choose my action, $a$.  The round node $\nu_a$ indicates the
generation of the stochastic outcome variable, $Y$, whose hypothetical
distribution $P_a$ will typically depend on the chosen action $a$.


Since, at node $\nu_a$, $Y \sim P_a$, the (negative) value of taking
action $a$, and thus getting to $\nu_a$, is measured by the expected
loss $L(a) := E_{Y\sim P_a} \{L(Y,a)\}$.  The principles of
statistical decision analysis now require that, at the decision node
$\nu_*$, I should choose an action $a$ minimising $L(a)$.


Note particularly that, whatever loss function is used, this solution
will only require knowledge of the collection $\{P_a\}$ of
hypothetical distributions for the outcome variable $Y$.

There are decision problems where explicit inclusion of the action $a$
as an argument of the loss function is natural.  For example, I might
have a choice between taking my umbrella ($a = 1$) when I go out, or
leaving it at home ($a = 0$).  For either action, the relevant binary
outcome variable $Y$ indicates whether it rains ($Y=1$) or not
($Y=0$).  The loss is 1 if I get wet, 0 otherwise, so that
$L(0,0) = L(0,1) = L(1,1) = 0$, $L(1,0) = 1$.  In this case, my action
presumably has no effect on the outcome $Y$, so that I might take
$P_1$ and $P_0$ to be identical; but it enters non-trivially into the
loss function.  However, it is arguable whether such a problem, where
the only effect of my action is on the loss, can properly be described
as one of causality.  In typical causal applications, the loss
function will depend only on the value $y$ of $Y$, and not further on
my action---so that $L(y,a)$ simplifies to $L(y)$.  The only thing
depending on $a$ will then be my hypothetical distribution $P_a$ for
$Y$, subsequent to (``caused by'') my taking action $a$.  Then
$L(a) = E_{Y\sim P_a} \{L(Y)\}$, and my choice of action effectively
becomes a choice between the different hypothetical distributions
$P_a$ for $Y$ associated with my available actions $a$: I prefer that
distribution giving the smallest expectation for $L(Y)$.  This
specialisation will be assumed throughout this work.
 
\section{A simple causal decision problem}
\label{sec:simple}
As a simple specific example, we consider the following stylised
decision problem.

\begin{ex}
  \label{ex:aspirin}
  I have a headache and am considering whether or not I should take
  two aspirin tablets.  Will taking the aspirins cause my headache to
  disappear?

  Let the binary decision variable $F_X$ denote whether I take the
  aspirin ($F_X = 1$) or not ($F_X = 0$), and let $Z$ denote the time
  it takes for my headache to go away.  For convenience only, we focus
  on $Y := \log Z$, which can take both positive and negative values.

  I myself will choose the value of $F_X$: it is a decision variable,
  and does not have a probability distribution.  Nevertheless, it is
  still meaningful to consider my conditional distribution, $P_x$ say,
  for how the eventual response $Y$ would turn out, where I to take
  decision $F_X = x$ ($x=0,1$).  For the moment we assume the
  distributions $P_0$, $P_1$ to be known---this will be relaxed in
  \secref{populating}.  Where we need to be definite, we shall, purely
  for simplicity, take $P_x$ to have the normal distribution
  $\norm(\mu_x,\sigma^2)$, with probability density function:
\begin{equation}
  \label{eq:normal}
  p_x(y) \equiv p(y \cd F_X =  x) = (2\pi\sigma^2)^{-\half} \exp -\frac{(y - \mu_x)^2}{2\sigma^2},
\end{equation}
having mean $\mu_0$ or $\mu_1$ according as $x= 0$ or $1$, and variance
$\sigma^2$ in either case.
  
The distribution $P_1$ [resp., $P_0$] expresses my {\em
  hypothetical\/} uncertainty about how $Y$ would turn out, {\em if\/}
I were to decide to take the aspirin, \ie\ under $F_X = 1$ [resp., if
I were to decide not to take the aspirin, $F_X = 0$].  It can
incorporate various sources and types of uncertainty, including
stochastic effects of external influences arising or acting between
the point of treatment application and the eventual response.  My task
is to compare the two hypothetical distributions $P_1$ and $P_0$, and
decide which one I prefer.  If I prefer $P_1$ to $P_0$, then my
decision should be to take the aspirin; otherwise, not.  Whatever
criterion I use, all I need to put it into effect, and so solve my
decision problem, is the pair of hypothetical distributions
$\{P_0, P_1\}$ for the outcome $Y$, under each of my hypothesised
actions.

One possible comparison of $P_1$ and $P_0$ might be in terms their
respective means, $\mu_1$ and $\mu_0$, for $Y$; the ``effect'' of
taking aspirin, rather than nothing, might then be quantified by means
of the change in the expected response, $\delta:= \mu_1 - \mu_0$.
This is termed the {\em average causal effect\/}, \ace (in terms of
the outcome variable $Y$---so more specifically denoted by $\ace_Y$,
if required).  Alternatively, we might look at the average causal
effect in terms of $Z = e^Y$:
$\ace_Z = \E_{P_1}(Z)-\E_{P_0}(Z) =
e^{\sigma^2/2}(e^{\mu_1}-e^{\mu_0})$, or make this comparison as a
ratio, $\E_{P_1}(Z)/\E_{P_0}(Z) = e^{\mu_1-\mu_0}$.  Or, we could
consider and compare the variance of $Z$,
$\var_x(Z) = e^{2\mu_x}(e^{2\sigma^2}-e^{\sigma^2})$ under $P_x$
($x=0,1)$.  In full generality, any comparison of an appropriately
chosen feature of the two hypothetical distributions, $P_0$ and $P_1$,
of $Y$ can be regarded as a partial summary of the {\em causal
  effect\/} of taking aspirin (as against taking nothing).

A fully decision-theoretic formulation is represented by the decision
tree of \figref{1}.
\begin{figure}[htbp]
\begin{center}
  \resizebox{3in}{!}{\includegraphics{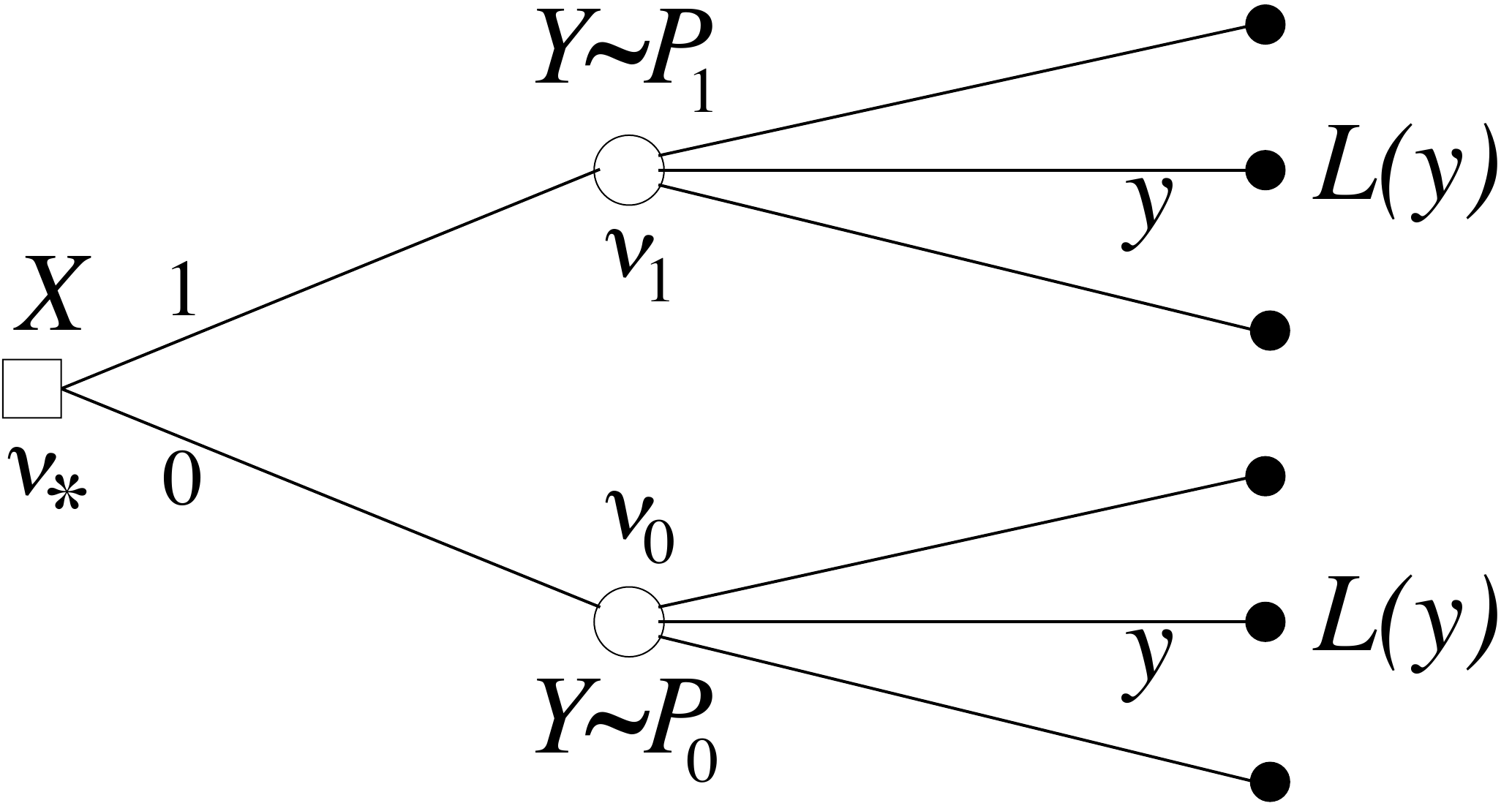}}
\caption{Decision tree}
\label{fig:1}
\end{center}
\end{figure}

Suppose (for example) that I were to measure the loss that I will
suffer if my headache lasts $z = e^y$ minutes by means of the
real-valued loss function $L(z) = \log z =y$.  If I were to take the
aspirin ($F_X = 1$), my expected loss would be
$E_{Y\sim P_1} (Y) = \mu_1$; if not ($F_X = 0$), it would be $\mu_0$.
The principles of statistical decision analysis now direct me to
choose the action leading to the smaller expected loss.  The ``effect
of taking aspirin'' might be measured by the increase in expected
loss, which in this case is just $\ace_Y$; and the correct decision
will be to take aspirin when this is negative.

Although there is no uniquely appropriate measure of ``the effect of
treatment'', in the rest of our discussion we shall, purely for
simplicity and with no real loss of generality, focus on the
difference of the means of the two hypothetical distributions for the
outcome variable $Y$:
\begin{equation}
  \label{eq:ace}
  \ace = \E_{P_1}(Y) - \E_{P_0}(Y).
\end{equation}
\end{ex}

\section{Populating the decision tree}
\label{sec:populating}
The above formulation is fine so long as I know all the ingredients in
the decision tree, in particular the two hypothetical distributions
$P_0$ and $P_1$.  Suppose, however, that I am uncertain about the
parameters $\mu_1$ and $\mu_0$ of the relevant hypothetical
distributions $P_1$ and $P_0$ (purely for simplicity we shall continue
to regard $\sigma^2$ as known).  To make explicit the dependence of
the hypothetical distributions on the parameters, we now write them as
$P_{1,\mu_1}$, $P_{0,\mu_0}$, and denote the associated density
functions by $p_{1}(y\cd \mu_1)$, $p_{0}(y \cd \mu_0)$.

\subsection{No-data decision problem}
\label{sec:nodata}

Being now uncertain about the parameter-pair $\bmu = (\mu_1, \mu_0)$,
I should assess my personalist prior probability distribution, $\Pi$
say, for $\bmu$ (in the light of whatever information I currently
have).  Let this have density $\pi(\mu_1,\mu_0)$.  To solve my
decision problem, I would then substitute, for the unknown
hypothetical distribution $P_{1,\mu_1}(y)$, my ``prior predictive''
hypothetical distribution $P^*_1$ for $Y$, with density
\begin{eqnarray*}
  p^*_1 (y) &=& \int\int p_{1}(y \cd \mu_1)\, \pi(\mu_1,\mu_0)\,d\mu_1\,d\mu_0\\
            &=& \int p_{1}(y \cd \mu_1)\, \pi_1(\mu_1)\,d\mu_1
\end{eqnarray*}
where $\pi_1(\mu_1)$ is my marginal prior density for $\mu_1$:
\begin{displaymath}
  \pi_1(\mu_1) = \int \pi(\mu_1,\mu_0)\,d\mu_0.
\end{displaymath}
Similarly, I would replace $P_{0,\mu_0}(y)$ by $P_0^*$, having density
\begin{math}
  p^*_0 (y) = \int p_{0}(y \cd \mu_0)\, \pi_0(\mu_0)\,d\mu_0,
\end{math}
where
\begin{math}
  \pi_0(\mu_0) = \int \pi(\mu_1,\mu_0)\,d\mu_1
\end{math}
is my marginal prior density for $\mu_0$.  We remark that, in parallel
to the property that, with full information, I only need to specify
the two hypothetical distributions $P_1$ and $P_0$, when I have only
partial information I only need to specify, separately, my marginal
uncertainties about the unknown parameters of each of these
distributions.  In particular, once these margins have been specified,
any further dependence structure in my joint personal probability
distribution $\Pi$ for $(\mu_1,\mu_0)$ is irrelevant to my decision
problem.

\subsection{Data}
\label{sec:data}
When in a state of uncertainty, that uncertainty can often be reduced
by gathering data.  Bayesian statistical decision theory
\cite{raiffa:schlaifer:61} shows that, for any decision problem, the
expected reduction in loss by using additional data (``the expected
value of sample information'') is always non-negative.  The effect of
obtaining data $D$ is to replace all the distributions entering in
\secref{nodata} above by their versions obtained by further
conditioning on $D$.

Suppose then that I wish to reduce my uncertainty about $\mu_1$, the
parameter of my hypothetical distribution $P_1$, by utilising relevant
data.  What data should I collect, and how should I use them?

What I might, ideally, want to do is gather together a ``treatment
group'' ${\cal T}$ of individuals whom I can regard, in an intuitive
sense, as similar to myself, with headaches similar to my own.  We
call such individuals {\em exchangeable\/} (both with each other and
with me)---this intuitive concept is treated more formally in
\secref{exchpred} below.  I then give them each two aspirins, and
observe their responses (how long until their headaches go away).
Conditionally on the parameter $\mu_1$ of $P_1 = P_{1,\mu_1}$, I could
reasonably\footnote{See \secref{exchpred} for formal justification.}
model these responses as being independently and identically
distributed, with the same distribution, $P_{1,\mu_1}$, that would
describe my own uncertainty about my own outcome, $Y$, were I,
hypothetically, to take the aspirins, and thus put myself into the
identical situation as the individuals in my sample.  Conditionally on
$\mu_1$, I would further regard my own outcome as independent of those
in the sample.  We shall not here be concerned with issues of sampling
variability in finite datasets.  So we consider the case that the
treatment group ${\cal T}$ is very large.  Then I can essentially
identify $\mu_1$ as the observed sample mean $\widehat\mu_1$, and so
take my updated $P_1$ to be
$\norm(\widehat\mu_1,\sigma^2)$.\footnote{This is of course specific
  to our assumed normal model \eqref{normal}, and in any case assumes
  $\sigma^2$ known.  For other models we might plug in the maximum
  likelihood estimate (or any other consistent estimate).  Still more
  generally, we could estimate the distribution $P_1$ nonparametricly,
  \eg\ using the empirical distribution of the sample data.}  For any
non-dogmatic prior, this will be a close approximation to my Bayesian
``posterior predictive distribution'' for $Y$, given the data $D$
(conditionally on my taking the aspirins), and also has a clear
frequentist justification.

The above was relevant to my hypothetical distribution $P_1$, were I
to take the aspirins.  But of course an entirely parallel argument can
be applied to estimating $P_0$, the distribution of my response $Y$
were I not to take the aspirins.  I would gather another large group
(the ``control group'', ${\cal C}$) of individuals similar to myself,
with headaches similar to my own, but this time withhold the aspirins
from them.  I would then use the empirically estimated distribution of
the response in this group as my own distribution $P_0$.

Let ${\cal D} = {\cal T}\cup {\cal C}$ be the set of ``data
individuals''.  Using the responses of ${\cal D}$, I have been able to
populate my own decision problem with the relevant hypothetical
distributions, $P_1$ and $P_0$.  I can now solve it, and so choose the
optimal decision for me.

\section{Exchangeability}
\label{sec:exchpred}
Here we delve more deeply into the justification for some of the
intuitive arguments made above (and below). 

In \secref{data} above, in the context first of estimating my
hypothetical distribution $P_1$, we talked of constructing, as the
treatment group ${\cal T}$,
\begin{quote}
  ``a group of individuals whom I can regard, in an intuitive sense,
  as similar to myself, with headaches similar to my own''.
\end{quote}
The identical requirement was imposed on the control group ${\cal C}$.
The formal definition and theory of {\em exchangeability\/}
\cite{definetti:37,definetti:75} seeks to put this intuitive
conception on a more formal footing.

We consider a collection ${\cal I}$ of individuals, on each of which
we can measure a number of generic variables.  One such is the generic
{\em response variable\/} $Y$, having a specific instance, $Y_i$, for
individual $i$---that is, $Y_i$ denotes the response of individual
$i$.  We suppose all individuals considered are included in
${\cal I}$.  In particular, ${\cal T}\subseteq {\cal I}$,
${\cal C}\subseteq {\cal I}$, and I myself am included in ${\cal I}$,
with label $0$, say.

\subsection{Post-treatment exchangeability}
\label{sec:postt}
What we are essentially requiring of ${\cal T}$, in the description
quoted above, is twofold:
\begin{enumerate}
\item \label{it:exi} My joint personalist distribution for the
  responses in the treatment group, \ie\ the ordered set
  $(Y_i:i\in{\cal T})$, is exchangeable---that is to say, I regard the
  re-ordered set $(Y_{\rho(i)}:i\in{\cal T})$ as having the same joint
  distribution as $(Y_i:i\in{\cal T})$, where $\rho$ is an arbitrary
  permutation (re-ordering) of the treated individuals.
\item \label{it:ex0} {\em If\/}, moreover, I were to take the
  aspirins, then the above exchangeability would extend to the set
  ${\cal T}^+ := {\cal T}\cup\{0\}$, in which I too am included.
\end{enumerate}
Parallel exchangeability assumptions would be made for the control
group ${\cal C}$, from whom the aspirin is withheld: in \itref{exi}
and \itref{ex0} we just replace ``treatment'' by ``control'',
${\cal T}$ by ${\cal C}$ (and ${\cal T}^+$ by ${\cal C}^+$), and
``were to take'' by ``were not to take''.  We shall denote these
variant versions by \itref{exi}$'$ and \itref{ex0}$'$.

Since the above exchangeability assumptions relate to the responses of
individuals after they have (actually or hypothetically) received
treatment, we refer to them as {\em post-treatment exchangeabiity\/}.

Applying de Finetti's representation theorem \cite{definetti:37} to
\itref{exi}, I can regard the responses $(Y_i:i\in{\cal T})$ in the
treatment group as independently and identically distributed, from
some unknown distribution.\footnote{Strictly, this result requires
  that I could, at least in principle, extend the size of the
  treatment group indefinitely, while retaining exchangeability.}
This distribution can then be consistently estimated from the response
data in the treatment group.  On account of \itref{ex0}, this same
distribution would govern my own response, $Y_0$, were I to take the
aspirins.  It can thus be identified with my own hypothetical
distribution $P_1$.  Taken together, \itref{exi} and \itref{ex0} thus
justify my estimating $P_1$ from the treatment group data, and using
this to populate the treatment branch of my decision tree.%
\footnote{More correctly, I should take account of all the data, in
  both groups.  I regard the associated ordered outcomes as {\em
    partially exchangeable\/} \cite{def:partial}, with a joint
  distribution unchanged under arbitrary permutations of individuals
  within each group.  Such a joint distribution can be regarded as
  generated by independent sampling, from a distribution $P_1$ for an
  individual in the treatment group, or $P_0$ for an individual in the
  treatment group, where I have a joint distribution for the pair
  $(P_0,P_1)$.  There could be dependence between $P_0$ and $P_1$ in
  this joint distribution (for example, they might contain common
  parameters)---in which case data on responses in the control group
  could also carry information about the treatment response
  distribution $P_1$.  Nevertheless, if the treatment data are
  sufficiently extensive I can still estimate $P_1$ consistently by
  ignoring the control data, and so use just the treatment data to
  populate the treatment arm of my decision problem.}  Similarly,
using \itref{exi}$'$ and \itref{ex0}$'$, I can use the data from the
control group to populate my own control branch.  My decision problem
can now be solved.\footnote{\label{fn:coherence} The above argument
  glosses over a small philosophical problem: Can I justify equating
  the {\em hypothetical\/} uncertainty about the response $Y$, {\em
    were an individual to take\/} the aspirins, with the {\em
    realised\/} uncertainty about (still unobserved) $Y$, once {\em
    that individual is known to have taken\/} the aspirins?\ (and,
  importantly, nothing else new is known).  The former is what is
  relevant to my decision problem, but the data on the treated
  individuals are informative about the latter.  We have implicitly
  assumed that these uncertainties are the same, and so governed by
  the same distribution.  We may term this property {\em temporal
    coherence\/}.  At a fully general level, any conditional
  probability $P(A \cd B)$ has two different interpretations: the
  (hypothetical) probability it would be appropriate to assign to $A$,
  were $B$ (and only $B$) to become known, and the (realised)
  probability it is appropriate to assign to $A$, after $B$ (but
  nothing else new) has become known.  Although it seems innocuous to
  equate these two, a full philosophical justification is not entirely
  trivial (see for example \textcite{skyrms:87}).  Nevertheless there
  is no serious dissent from this position, and we shall adopt it
  without further ado.}

\paragraph{Some comments}
\renewcommand{\theenumi}{(\arabic{enumi})}
\begin{enumerate}
\item Whether or not the exchangeability assumption \itref{exi} can be
  regarded as reasonable will be highly dependent on the background
  information informing my personal probability assessments.  For
  example, I might know, or suspect, that evening headaches tend to be
  more long-lasting than morning headaches.  If I were also to know
  which of the headaches in ${\cal T}$ were evening, and which
  morning, headaches, then I would not wish to impose exchangeability.
  I might know that individual 1 had a morning headache, and
  individual 2 an evening headache.  Then it would not be reasonable
  for me to give the re-ordered pair $(Y_2,Y_1)$ the same joint
  distribution as $(Y_1,Y_2)$---in particular, my marginal
  distribution for $Y_2$ would likely not be the same as that for
  $Y_1$.  However, in the absence of specific knowledge about who had
  what type of headache---``equality of ignorance''---the
  exchangeability condition \itref{exi} could still be reasonable.

\item There may be more than one way of embedding my own response,
  $Y_0$, into a set of exchangeable variables.  For example, instead
  of considering other individuals, I could consider all my own
  previous headache episodes.  (In the language of experimental
  design, the experimental unit---the headache episode---is nested
  within the individual).  Then I might use the estimated distribution
  of my response, among those past headache episodes of my own that I
  had treated with aspirin, to populate the treatment branch of my
  current decision problem.  This might well yield a different (and
  arguably more relevant) distribution from that based on observing
  headaches in other treated individuals.  In this sense there is no
  ``objective'' distribution $P_1$ waiting to be uncovered: $P_1$ is
  itself an artifact of the overall structure in which I have embedded
  my problem, and the data that I have observed.

\item Exchangeability must also be considered in relation to my own
  current circumstances.  The exchangeability judgment \itref{exi} may
  not be extendible as required by \itref{ex0} if, for example, my
  current headache is particularly severe.  To reinstate
  exchangeability I might then need to restrict attention to those
  headache episodes (in other individuals, or in my own past) that had
  a similar level of severity to mine.  Alternatively I might build a
  more complex statistical model, allowing for different degrees of
  severity, and use this to extrapolate from the observed data
  to my own case.

\item \label{it:vaccine} We do not in principle exclude complicated
  scenarios such as ``herd immunity'' in vaccination programmes, where
  an individual's response might be affected in part by the treatments
  that are assigned to other individuals.  Assuming appropriate
  symmetry in (my knowledge of) the interactions between individuals,
  this need not negate the appropriateness of the exchangeability
  assumptions, and hence the validity of the above analysis---though
  in this case it would be difficult to give the underlying
  distributions $P_0$ and $P_1$, conjured into existence by
  de~Finetti's theorem, a clear frequentist interpretation.  However,
  in such a problem it would usually be more appropriate to enter into
  a more detailed modelling of the situation.
\end{enumerate}

Exchangeability, while an enormously simplifying assumption, is in any
case inessential for the more general analysis of \secref{data}: at
that level of generality, I have to assess my conditional distribution
for my own response $Y_0$ (in the hypothetical situation that I decide
to take the aspirins), given whatever data $D$ I have available.  But
modelling and implementing an unstructured prediction problem can be
extremely challenging, as well as hard to justify as genuinely
empirically based, unless we can make good arguments.  When
appropriate, judgments of exchangeability constitute an excellent
basis for such arguments.

\subsection{Pre-treatment exchangeability}
\label{sec:extexch}
The post-treatment exchangeability conditions \itref{exi} and
\itref{ex0}, and \itref{exi}$'$ and \itref{ex0}$'$, are what is needed
to let me populate my decision tree with the requisite hypothetical
distributions and so solve my decision problem.

Here we consider another interpretation of the expression ``a group of
individuals whom I can regard, in an intuitive sense, as similar to
myself, with headaches similar to my own''.  This description has been
supposed equally applicable to the treatment group ${\cal T}$ and the
control group ${\cal C}$.  But this being the case, then---applying
Euclid's first axiom, ``Things which are equal to the same thing are
also equal to one another''---the two groups, ${\cal T}$ and
${\cal C}$ (and their headaches), both being similar to me, must be
regarded (again in an intuitive sense) as similar to each other---I
must be ``comparing like with like''.  But how are we to formalise
this intuitive property of the two groups being similar to each other?
We cannot simply impose full exchangeability of all the responses
$(Y_i: i\in{\cal D})$, since I typically would not expect the
responses of the treated individuals to be exchangeable with those of
the untreated individuals.

One way of formalising this intuition is to consider all the
individuals in the treatment and control groups {\em before\/} they
were given their treatments.  Just as I myself can hypothesise taking
either one of the treatments, and in either case consider my
hypothetical distribution for my ensuing response $Y_0$, so can I
hypothesise various ways in which treatments might be applied to all
the individuals in ${\cal I}$.

Let the binary decision variable $\check T_i$ indicate which treatment
is hypothesised to be applied to individual $i$.

We first introduce the following {\em Stable Unit-Treatment
  Distribution Assumption\/}:
\begin{cond}[SUTDA]
  \label{cond:sutda}
  For any $A \subseteq {\cal I}$, the joint distribution of
  $Y_A := (Y_i:Y\in A)$, given hypothesised treatment applications
  $(\check T_i = t_i: i\in{\cal I})$, depends only on $(t_i:i\in A)$.
  In particular, for any individual $i$, the distribution of the
  associated response $Y_i$ depends only on the treatment $t_i$
  applied to that individual.
\end{cond}
As discussed further in \secref{po} below, SUTDA bears a close
resemblance to the Stable Unit-Treatment {\em Value\/} Assumption,
SUTVA, typically made in the Rubin potential outcome framework;
but---as reflected in it name---differs in the important respect of
referring to distributions, rather than values, of variables.  It is a
weaker requirement than SUTVA, but is as powerful as required for
applications.

Note that SUDTA is a genuinely restrictive hypothesis, now excluding
cases such as the vaccine example \itref{vaccine} of
\secref{exchpred}.  However, we will henceforth assume it holds.

In more complex problems there will be other generic variables of
interest besides $Y$---we term these (including the response variable
$Y$) {\em domain variables\/}.  Then we extend SUTDA to apply to all
domain variables, considered jointly.  An important special case is
that of a domain variable $X$ such that the joint distribution of
$(X_i:i\in{\cal I})$, given $\check T_i = t_i$ ($i\in{\cal I})$, does
not depend in any way on the applied treatments $(t_i)$.  Such a
variable, unaffected by the treatment, is a {\em concomitant\/}.  It
will typically be reasonable to treat as a concomitant any variable
whose value is fully determined before the treatment decision has to
be made: such a variable is termed a {\em covariate\/}.  Other
concomitants might include, for example, the weather after the
treatment decision is made.

Let $V$ be a (possibly multivariate) generic variable.  I now
hypothesise giving {\em all\/} individuals in ${\cal I}$ (including
myself) the aspirins, and consider my corresponding hypothetical joint
distribution for the individual instances $(V_i: i\in{\cal I})$.  It
would often be reasonable to impose full exchangeability on this joint
distribution, since all members of ${\cal I}$ would have been treated
the same.  A similar assumption can be made for the case that the
aspirins are, hypothetically, withheld from all individuals.  We term
the conjunction of these two hypothetical exchangeability properties
{\em pre-treatment exchangeability\/} (of $V$, over ${\cal I}$).

When I can assume this, then under uniform application of aspirin, by
de~Finetti's theorem I can regard all the $(V_i)$ as independent and
identically distributed from some distribution $Q_1$ (initially
unknown, but estimable from data on uniformly treated individuals).
Similarly, under hypothetical uniform withholding of aspirin, there
will be an associated distribution $Q_0$.  When moreover SUTDA
applies, we can conclude that, under any hypothesised application of
treatments, $\check T_i = t_i$ ($i\in{\cal I})$, we can regard the
$V_i$ as independent, with $V_i \sim Q_{t_i}$.  We can thus confine
attention to the generic variable $V$, with distribution $Q_1$ [resp.,
$Q_0$] under applied treatment $\check T = 1$ [resp., $\check T = 0$].

Pre-treatment exchangeability appears, superficially, to be a stronger
requirement than post-treatment exchangeability: one could argue that
(taken together with SUDTA) pre-treatment exchangeability implies the
post-treatment exchangeability properties \itref{exi}, \itref{ex0},
\itref{exi}$'$ and \itref{ex0}$'$, which would permit me to populate
both the treatment and the control branches of my decision tree, and
so solve my decision problem.  This would indeed be so if the
individuals forming the treatment and control groups were identified
in advance, and then subjected to their appointed interventions.
However, it need not be so in the more general case that we do not
have direct control over who gets which treatment.  Much of the rest
of this paper is concerned with addressing such cases, considering
further conditions---in particular, {\em ignorability\/} of the
treatment assignment process, as described on \secref{ignore}
below---that allow us to bridge the gap between pre- and
post-treatment exchangeability.

\subsection{Internal and external validity}
\label{sec:valid}
We might be willing to accept pre-treatment exchangeability, but only
over the restricted set ${\cal D}$ of data individuals, excluding
myself---a property we term {\em internal exchangeability\/}.  When I
can extend this to pre-treatment exchangeability over the set
${\cal D}^+ := {\cal D}\cup\{0\}$, including myself, we have {\em
  external exchangeability\/}.  In the latter case there is at least a
chance that the data ${\cal D}$ could help me solve my decision
problem---the case of {\em external validity\/} of the
data.\footnote{This is admittedly a very strict interpretation of
  ``external validity''.  More generally, it might be considered
  enough to be able to transfer information about, say, \ace, from the
  data to me.  This would typically require further modelling
  assumptions, such as described in \S8.1 of \textcite{apd:cinfer}.}
However, when we have internal but not external exchangeability, this
conclusion could, at best, be regarded as holding for a new, possibly
fictitious, individual who could be regarded as exchangeable with
those in the data---this is the case of {\em internal validity\/}.  In
practice that can be problematic.  For example, a clinical trial might
have tightly restricted enrolment criteria, perhaps restricting entry
to, say, men aged between 25 and 49 with mild headache.  Even if the
study has good internal validity, and shows a clear advantage to
aspirin for curing the headache, it is not clear that this message
would be relevant to a 20-year old female with a severe headache.  And
indeed, it may not be.  Arguments for external validity will generally
be somewhat speculative, and not easy to support with empirical
evidence.

\section{Treatment assignment and application}
\label{sec:assig}

In \secref{data} we talked in terms of identifying, quite separately,
two groups of individuals, in each case supposed suitably exchangeable
(both internally, and with me), where one of the groups is made to
take, and the other made not to take, the aspirins.  But typically the
process is reversed: a single group of individuals, ${\cal D}$ say, is
gathered, some of whom are then chosen to receive active
treatment---thus forming the treatment group ${\cal T}$---with the
remainder forming the control group ${\cal C}$.

In this case the treatment process has three stages:
\begin{enumerate}
\item \label{it:First} First, the data subjects ${\cal D}$ are
  identified by some process.
\item \label{it:Secondly} Secondly, certain individuals in ${\cal D}$
  are somehow selected to receive active treatment, the others
  receiving control.\footnote{In reality stages~\itref{First} and
    \itref{Secondly} may be combined, as in sequential accrual and
    randomisation in a clinical trial.}
\item \label{it:Finally} Finally, the assigned treatments are actually
  administered.
\end{enumerate}

The operation of stage~\itref{First} will be crucial for issues of
external validity---if the data are to be at all relevant for me, I
would want the data subjects to be somehow like me.  However from this
point on we shall na\"ively assume this has been done
satisfactorily---alternatively, we consider ``me'' to be a possibly
fictitious individual who can be regarded as similar to those in the
data.  We shall thus consider all data subjects, together with myself,
as pre-treatment exchangeable.  I can then confine attention to the
joint distributions $P_1$ and $P_0$ over generic variables, under
hypothesised application of treatment 1 or 0, respectively.

For further analysis it will prove important to keep
stages~\itref{Secondly} and \itref{Finally} clearly distinct in the
notation and the analysis.

We denote by $T^*$ the generic {\em intention to treat\/} (ITT)
variable, generated at stage~\itref{Secondly}, where $T^*_i=1$ if
individual $i\in{\cal D}$ is selected to receive active treatment, and
$T^*_i=0$ if not (this is relevant only for the external data
${\cal D}$: my own value $T^*_0$ need not be defined).  Note that
$T^*$ is a stochastic variable.  In contrast, we also consider (at
stage~\itref{Finally}) the binary non-stochastic generic
decision/regime variable $\check T$: $\check T_i = 1$ [resp.,
$\check T_i = 0$] denotes the (typically hypothetical) situation in
which individual $i$ is made to take [resp., prevented from taking]
the aspirins.  My own decision variable $\check T_0$ (though not yet
its value) is well-defined---indeed, is the very focus of my decision
problem.

Note that when below we talk of ``domain variables'' we will exclude
$T^*$ and $\check T$ from this description.

If all goes to plan, for $i\in{\cal D}$ we shall have
$\check T_i = T^*_i$.  However, there is no bar to considering,
between stages~\itref{Secondly} and \itref{Finally}, what might happen
to an individual, fingered to receive the treatment (so having
$T^*_i = 1$), who, contrary to plan, is prevented from taking it (so
that $\check T_i = 0$)%
\footnote{This apparently oxymoronic combination has some superficial
  resemblance to counterfactual reasoning (see \eg\
  \textcite{morgan_winship_2014}), which has often been
  considered---quite wrongly in my view \cite{apd:cinfer}---as
  essential for modelling and manipulating causal relations.
  Counterfactual analysis considers the individual after he has been
  treated (so with known $\check T_i = 1$, and possibly known response
  $Y_i$), and asks what might have happened if---in a fictional
  scenario {\em counter to known facts\/}---he had not been treated
  (\ie, under the counterfactual application $\check T_i = 0$).  In
  spite of some parallels, there are important differences between our
  {\em hypothetical\/} approach and this {\em counterfactual\/}
  approach.  By considering a time before any treatment has yet been
  applied, and making the distinction between intention to treat,
  $T^*_i$, and a hypothesised treatment application, $\check T_i$, we
  sidestep many of the philosophical and methodological difficulties
  associated with counterfactual reasoning.  In particular, in our
  formulation we avoid counterfactual theory's problematic and
  entirely unnecessary conversion of the single response variable $Y$
  into two separate but co-existing ``potential responses'', $Y(0)$
  and $Y(1)$.}---indeed, we have already made use of such
considerations when introducing pre-treatment exchangeability.  So we
can meaningfully consider a quantity such as
$\E(Y \cd T^* =1, \check T = 0)$.  And indeed it will prove useful to
divorce treatment {\em selection\/} (intention to treat), $T^*$, from
(actual or hypothetical) treatment {\em application\/}, $\check T$, in
this way.  For example, what is usually termed {\em the effect of
  treatment on the treated\/} \cite{heckman:socpol} is more properly
expressed as {\em the effect of treatment on those selected for
  treatment\/}, which can be represented formally as
$\E(Y \cd T^* =1, \check T = 1) - \E(Y \cd T^* =1, \check T = 0)$
\cite{sgg/apd:ett}.

Since the selection process is made before any application of
treatment, it is appropriate to treat $T^*$ as a covariate, with the
same distribution in both regimes.

We suppose internal exchangeability, in the sense of \secref{extexch}
above, for the pair of generic variables $(T^*,Y)$.  In particular we
shall have internal exchangeability, marginally, for the response
variable $Y$---and, to make a link to my own decision problem, we
assume this extends to external exchangeability for $Y$ (we here omit
$T^*$, since that might not even be meaningfully defined for me).
However, even internal exchangeability for $Y$ need no longer hold
after we condition on the selection variable $T^*$---this is the
problem of {\em confounding\/}.  For example, suppose that, although I
myself don't know which of the headaches in ${\cal D}$ are the
(generally milder) morning and which the (generally more long-lasting)
evening headaches, I know or suspect that the aspirins have been
assigned preferentially to the evening headaches.  Then simply knowing
that an individual was selected (perhaps self-selected) to take the
aspirins ($T^* = 1)$ will suggest that his headache is more likely to
be an evening headache, and so change my uncertainty about his
response $Y$ (whichever treatment were to be taken). I might thus
expect, \eg,
$\E(Y \cd T^* = 1, \check T = t) > \E(Y \cd T^* = 0, \check T = t)$,
both for $t=0$ and for $t=1$.
In such a case, even under a hypothetical uniform {\em application\/}
of treatment, I could not reasonably assume exchangeability between
the group {\em selected\/} to receive active treatment (and thus more
likely to have long-lasting evening headaches) and the group selected
for control (who are more likely to have short-lived morning
headaches).  Post-treatment exchangeability is absent, since I would
no longer be comparing like with like.  This in turn renders external
validity impossible, since (even under uniform treatment) I could not
now be exchangeable, simultaneously, both with those selected for
treatment and with the those selected for control, since these are not
even exchangeable with each other.  This means I can no longer use the
data (at any rate, not in the simple way considered thus far) to fully
populate, and thus solve, my decision problem.

As explained in \secref{extexch}, assuming internal exchangeability
and SUTDA, I can just consider the joint distribution, $Q_t$, for the
bivariate generic variable $(T^*,Y)$, given $\check T=t$.  Since we
are treating the selection indicator $T^*$ as a covariate, its
marginal distribution will not depend on which hypothetical treatment
application is under consideration, and so will be the same under both
$Q_1$ and $Q_0$.  We can express this as the extended independence
property
\begin{equation}
  \label{eq:indst}
  \indo {T^*} {\check T},
\end{equation}
which says that the (stochastic) selection variable $T^*$ is
independent of the (non-stochastic) decision variable $\check T$.  We
denote this common distribution of $T^*$ in both regimes by $P^*$.

By the assumed external exchangeability of $Y$, the marginal
distribution of $Y$ under $Q_t$ is my desired hypothetical response
distribution, $P_t$.  However, in the absence of actual uniform
application of treatment $t$ to the data subjects (which in any case
is not simultaneously possible for both values of $t$), I may not be
able to estimate this marginal distribution.  In the data, the
treatment will have been applied in accordance with the selection
process, so that $\check T=T^*$, and the only observations I will have
under regime $\check T = 1$ (say) are those for which $T^*=1$.  From
these I can estimate the {\em conditional\/} distribution of $Y$,
given $T^*=1$. under $Q_1$---but this need not agree with the desired
{\em marginal\/} distribution $P_1$ of $Y$ under $Q_1$.\footnote{Note
  that I can not make use of the {\em conditional\/} distributions of
  $Y$ given $T^*$.  Typically I myself do not even have, let alone
  observe, a value for $T^*_0$.  And even in the special case that my
  value $T^*_0$ is well-defined, and I can assume external validity
  for the pair $(T^*,Y)$, I can at best estimate one of the two
  required conditional distributions.  Thus if I have been fingered
  for treatment, $T^*_0 = 1$, I would need the conditional
  distribution of $Y$ given $T^* = 1$ under $Q_t$, for both $t=0$ and
  $t=1$.  But for $t=0$ this will not be estimable from the data,
  since there were no data subjects who were fingered for treatment
  but did not receive it.}

\subsection{Ignorability}
\label{sec:ignore}

The above complication will be avoided when I judge that, both for
$t=1$ and for $t=0$, if I intervene to {\em apply\/} treatment
$\check T = t$ on an individual, the ensuing response $Y$ will not
depend on the {\em intended\/} treatment $T^*$ for that individual;
\ie, we have independence of $Y$ and $T^*$ under each $Q_t$.  This can
be expressed as the extended conditional independence property
\begin{equation}
  \label{eq:indys_t}
  \ind Y {T^*} {\check T}.
\end{equation}
When \eqref{indys_t} can be assumed to hold, we term the assignment
process {\em ignorable\/}.  In that case, my desired distribution for
$Y$, under hypothesised active treatment assigment ${\check T}=1$, is
the same as the conditional distribution of $Y$ given $T^*=1$ under
${\check T}=1$---which is estimable as the distribution of $Y$ in the
treatment group data.  Likewise, my distribution for $Y$ under
hypothesised control treatment is estimable from the data in the
control group.

The ignorability condition \eqref{indys_t} requires that the
distribution of an individual's response $Y$, under either applied
treatment, will not be affected by knowledge of which treatment the
individual had been fingered to receive---a property that would likely
fail if, for example, treatment selection $T^*$ was related to the
overall health of the patient.  Note that ignorability is not testable
from the available data, in which ${\check T}=T^*$.  For we would need
to test, in particular, that, for an individual taking actual
treatment $\check T = 1$, the distribution of $Y$ given $T^* = 1$ is
the same as that given $T^*=0$.  But for all such individuals in the
data we never have $T^*=0$, so can not make the comparison.  Hence any
assumption of ignorability can only be justified on the basis of
non-empirical considerations.  The most common, and most convincing,
basis for such a justification is when I know that the treatment
assignment process has been carried out by a randomising device, which
can be assumed to be entirely unrelated to anything that could affect
the responses; but I might be able to make a non-empirical arguments
for ignorability in some other contexts also.  Indeed, it would be
rash simply to assume ignorability without having a good argument to
back it up.

\section{The idle regime}
\label{sec:idle}
As a useful extension of the above analysis, we expand the range of
the regime indicator $\check T$ to encompass a further value, which we
term ``idle'', and denote by $\idle$---this indicates the
observational regime, where treatments are applied according to plan.
(This is relevant only for the data individuals, in ${\cal D}$: I
myself care only about the two interventions I am considering).  We
denote this 3-valued regime indicator by $F_T$.

Now $T^*$ is determined prior to any (actual or hypothetical)
treatment application, and behaves as a covariate.  It is thus
reasonable to assume that, under the observational regime
$F_T = \idle$, $T^*$ retains its fixed covariate distribution $P^*$.
And since this distribution is then the same in all three regimes, we
thus have
\begin{equation}
  \label{eq:tstarsame}
  \indo {T^*} {F_T}.
\end{equation}
This extends \eqref{indst} to include also the idle regime.

We now introduce a new stochastic domain variable $T$, representing
the treatment actually applied when following the relevant regime.
This is fully determined by the pair $(F_T, T^*)$, as follows:
\begin{defn}{\bf (Applied Treatment, $T$)}
  \label{def:appl}
  \renewcommand{\theenumi}{(\roman{enumi})}
  \begin{enumerate}
  \item \label{it:appl1}
    If $F_T = 0$ or $1$, then $T=F_T$
  \item \label{it:appl2}
    If $F_T=\idle$, then $T=T^*$.
  \end{enumerate}
  In particular, $T\sim P^*$ under $F_T=\idle$, while $T$ has a
  degenerate distribution at $t$ under $F_T= t$ ($t=0$ or $1$).
\end{defn}

In each of the three regimes we can observe both $T$ and $Y$.  In the
observational regime ($F_T=\idle$) we can also recover $T^*$, since
$T^*=T$.  However, $T^*$ is typically unobservable in the
interventional regimes, and may not even be defined for myself, the
case of interest.

To complete the distributional specification of the idle regime we
argue as follows.  Under $F_T = \idle$, the information conveyed by
learning $T=t$ is twofold, conveying both that the individual was
initially fingered to receive treatment $t$, \ie\ $T^* = t$, and that
treatment $t$ was indeed applied.  Hence for any domain variable $V$,
the conditional distribution of $V$ given $T=t$ (equivalently, given
$T^* = t$), under $F_T = \idle$, should be the same as that of $V$
given $T^*=t$, under the (real or hypothetical) applied treatment
$F_T= t$.  We express this property formally as:
\begin{defn}{\bf (Distributional Consistency)}
  \label{def:distcons}
  For any domain variable, or set of domain variables,
  $V$,\footnote{Note that \eqref{idledist} holds, automatically, on
    taking for $V$ the constructed variable $T$, since on each side
    the conditional distribution for $V\equiv T$ is the one-point
    distribution on the value $t$.}
  \begin{equation}
    \label{eq:idledist}
    V \cd (T=t, F_T = \idle)\,\, [= V \cd (T^*=t, F_T = \idle)] \,\approx\, V \cd (T^*=t, F_T = t)\qquad(t=0,1),
  \end{equation}
  where $\approx$ denotes ``has the same distribution as''.
\end{defn}
Distributional consistency is the fundamental property linking the
observational and interventional regimes.  It is our, weaker, version
of the (functional) consistency property usually invoked in the
potential outcome approach to causality---see \secref{po} below.  In
the sequel we shall take \eqref{idledist} for granted.





\begin{lemma}
  \label{lem:basicdag}
For any domain variable $V$,
  \begin{equation}
    \label{eq:basicdag}
    \ind V {F_T} {(T, T^*)}.
  \end{equation}
\end{lemma}

\begin{proof}
  We have to show that, for $t, t^*\in \{0,1\}$, it is possible to
  define a conditional distribution for $V$, given $T=t, T^*=t^*$,
  that applies in all three regimes.

  Let $\Pi_{t,t^*}$ denote the distribution of $V$ given $T^*=t^*$ in
  the interventional regime $F_T = t$.  This is well-defined in the
  usual case that the event $T^*=t^*$ has positive probability (this
  probability being the same in all regimes)---if not, we make an
  arbitrary choice for this distribution.

  Consider first the case $t=1$.

  \begin{enumerate}
  \item Since $T$ is non-random with value $1$ in regime $F_t=1$,
    $\Pi_{1,t^*}$ is also, trivially, the distribution of $V$ given
    $T=1, T^*=t^*$ in regime $F_T=1$.
  \item Under regime $F_T = 0$, the event $T=1,T^* = t^*$ has
    probability $0$, so we are free to define the distribution of $V$
    conditional on this event arbitrarily; in particular we can take
    it to be $\Pi_{1,t^*}$.
  \item Under regime $F_T=\idle$, the event $T=1, T^*=0$ has
    probability $0$, so we are free to define the distribution of $V$
    conditional on this event as $\Pi_{1,0}$.
  \item It remains to show that the distribution of $V$ given
    $T=T^* = 1$ in regime $F_T=\idle$ is $\Pi_{1,1}$.  Since, under
    $F_T = \idle$, $T\equiv T^*$, we need only condition on $T=1$.
    The result now follows from distributional consistency
    \eqref{idledist}.
  \end{enumerate}

  Since a parallel argument holds for the case $t=0$, we have shown
  that $\Pi_{t,t^*}$ serves as the conditional distribution for $V$
  given $(T=t, T^* = t^*)$ in all three regimes, and \eqref{basicdag}
  is thus proved.
\end{proof}

\subsection{Graphical representation}
\label{sec:graphrep}
The properties \eqref{tstarsame} and \eqref{basicdag} are represented
graphically (using $d$-separation) by the absence of arrows from $F_T$
to $T^*$ and to $Y$, respectively, in the ITT (intention to treat) DAG
of \figref{basic}, where again, a round node represents a stochastic
variable, and a square node a non-stochastic regime indicator.  In
addition, we have included further optional annotations:
\begin{itemize}
\item The outline of $T^*$ is dotted to indicate that $T^*$ is not
  directly observed
\item The heavy outline of $T$ indicates that the value of $T$ is {\em
    functionally\/} determined by those of its parents $F_T$ and $T^*$
\item The dashed arrow from $T^*$ to $T$ indicates that this arrow can
  be removed (there is then no dependence of $T$ on $T^*$) under
  either of the {\em interventional\/} settings $F_T = 0$ or $1$.
\end{itemize} 

\begin{figure}[htbp]
  \begin{center}
    \resizebox{2in}{!}{\includegraphics{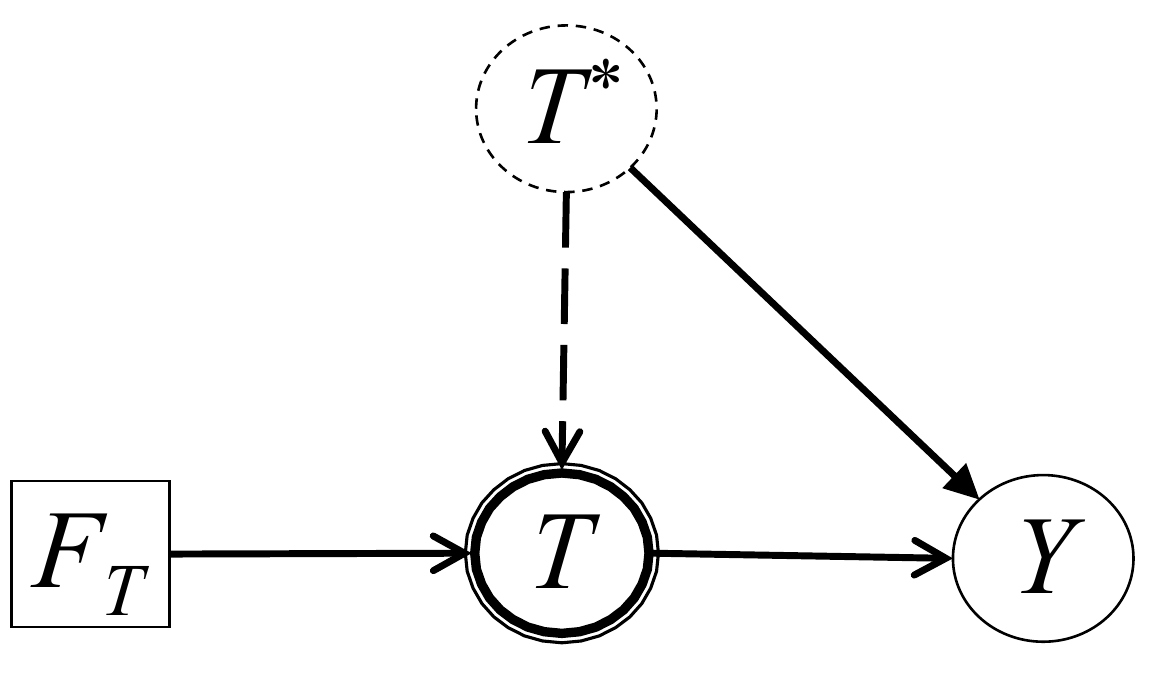}}
    \caption{DAG representing $\indo {T^*} {F_T}$ and
      $\ind Y {F_T} {(T, T^*)}$}
    \label{fig:basic}
  \end{center}
\end{figure}

\begin{rem}
  \label{rem:ok}
  Note that, on further taking into account the functional
  relationship of \defref{appl}, \figref{basic} already incorporates
  the distributional consistency property of \defref{distcons}, for
  $V\equiv Y$.  For we have
  \begin{eqnarray}
    \label{eq:ok1}
    Y \mid (T=t, F_T = \idle) &=&  Y \mid (T=t, T^* = t, F_T = \idle)\\
    \label{eq:ok2}
                              &\approx&  Y \mid (T=t, T^* = t, F_T = t)\\
    \label{eq:ok3}
                              &=&  Y \mid (T^* = t, F_T = t).
  \end{eqnarray}
  Here \eqref{ok1} follows from \itref{appl2} of \defref{appl};
  \eqref{ok2} from \lemref{basicdag} with $V \equiv Y$, \ie\
  $\ind Y {F_T} {(T,T^*)}$, which is represented in \figref{basic};
  and \eqref{ok1} from \itref{appl1} of \defref{appl}.
\end{rem}

Now the ITT variable $T^*$, while crucial to understanding the
relationship between the different regimes, is not itself directly
observable.  If we confine attention to relationships between the
domain variables, \figref{basic} collapses into the essentially
vacuous DAG of \figref{collapse}, expressing no non-trivial
conditional independence properties.
\begin{figure}[htbp]
  \begin{center}
    \resizebox{2in}{!}{\includegraphics{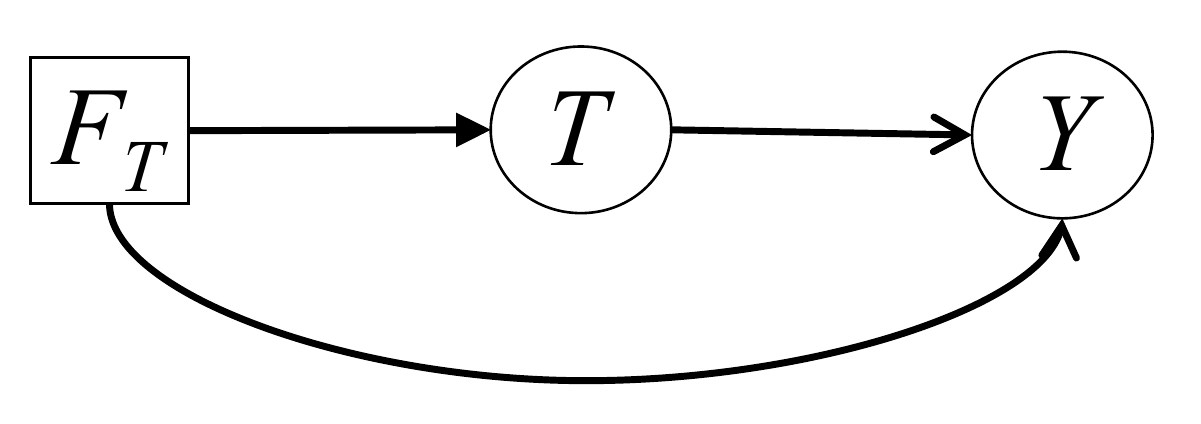}}
    \caption{Vacuous DAG between $F_T$, $T$ and $Y$, ignoring $T^*$}
    \label{fig:collapse}
  \end{center}
\end{figure}
So without further assumptions there is no useful structure of which
to avail ourselves.

\subsection{Ignorability}\label{sec:ignalg}
Suppose now we impose the additional ignorability property
\eqref{indys_t}.  Noting that $\check T=t$ is identical with $F_T=t$,
this is equivalent to
\begin{equation}
  \label{eq:indyft}
  \ind Y {T^*} {F_T=t}\qquad(t=0,1).
\end{equation}
Equivalently, since $T$ is non-random in an interventional regime,
\begin{displaymath}
  \ind Y {T^*} {(T,F_T=t)}\qquad(t=0,1).
\end{displaymath}
Moreover, since in the idle regime, $T^*$ is identical with $T$, so
non-random when $T$ is given, we trivially have
\begin{displaymath}
  \ind Y {T^*} {(T, F_T=\idle)}.
\end{displaymath}
We thus see that ignorability can be expressed as:
\begin{equation}
  \label{eq:ignf}
  \ind Y {T^*} {(T, F_T)}.
\end{equation}

\begin{lemma}
  \label{lem:ignalg}
  If ignorability holds, then
  \begin{equation}
    \label{eq:igndag}  
    \ind Y {F_T} T.
  \end{equation}
\end{lemma}

\begin{proof}
  We first dispose of the trivial case that $T^*$ has a one-point
  distribution.  In that case the conditioning on $T^*$ in
  \eqref{basicdag} is redundant and we immediately obtain
  \eqref{igndag}.

  Otherwise, $0<\pr(T^*=1)<1$.  We then have
  \begin{eqnarray}
    \label{eq:ig1}        
    Y \mid (T=1, F_T = \idle) &\approx&  Y \mid ( T^* = 1, F_T = 1)\\
    \label{eq:ig2}
                              &\approx& Y \mid F_T = 1\\
    \label{eq:ig3}
                              &\approx& Y \mid (T=1, F_T = 1).
  \end{eqnarray}
  Note that all conditioning events have positive probability in their
  respective regimes.  Here \eqref{ig1} holds by distributional
  consistency \eqref{idledist}, \eqref{ig2} by ignorability
  \eqref{indyft}, and \eqref{ig3} because, under $F_T = 1$, $T=1$ with
  probability 1.  So we have a common well-defined distribution,
  $\Delta_1$ say, for $Y$ given $T=1$ in both regimes $F_T = \idle$
  and $F_T = 1$.  Further, since under $F_T=0$ the event $T=1$ has
  probability $0$, we are free to define the conditional distribution
  of $Y$ given $T=1$ in regime $F_T=0$ as $\Delta_1$ also, so making
  $\Delta_1$ the common distribution of $Y$ given $T=1$ in all 3
  regimes, showing that $\ind Y {F_T} {T=1}$.  Since a similar
  argument holds for conditioning on $T=0$ the result follows.
\end{proof}
\begin{rem}
  \label{rem:altalg}
  An apparently simpler alternative proof of \lemref{ignalg} is as
  follows.  By \lemref{basicdag}, the conditional distribution of $Y$,
  given $({F_T}, T, T^*)$, does not depend on $F_T$, while by
  \eqref{ignf} this conditional distribution does not depend on $T^*$.
  So (it appears), it must follow that it depends only on $T$, whence
  $\ind Y {(F_T, T^*)} T$, implying the desired result.  This is a
  special case of a more general argument: that $\ind X Y {(Z, W)}$
  and $\ind X Z {(Y, W)}$ together imply $\ind X {(Y,Z)} W$.  However
  this argument is invalid in general \cite{apd:misl}.  To justify it
  in this case we have needed, in our proof \lemref{ignalg}, to call
  on structural properties (in particular, distributional consistency,
  and the way in which $T$ is determined by $F_T$ and $T^*$) in
  addition to conditional independence properties.
\end{rem}
\begin{cor}
  \label{cor:aaa}
  Ignorability holds if and only if
  \begin{equation}
    \label{eq:ddd}
    \ind Y {(T^*,F_T)} T.
  \end{equation}    
\end{cor}
\begin{proof}
  \begin{description}
  \item[If:] Further conditioning \eqref{ddd} on $F_T$ yields
    \eqref{ignf}.
  \item[Only if:] Property~\eqref{ddd} is equivalent to the conjunction
    of \eqref{ignf} and \eqref{igndag}.
  \end{description}
\end{proof}


\subsubsection{Graphical representation}
\label{sec:ignore2}
The DAG representing \eqref{tstarsame} and \eqref{ddd} is shown in
\figref{igndag}.  Compared with \figref{basic}, we see that the arrow
from $T^*$ to $Y$ has been removed.
\begin{figure}[htbp]
  \begin{center}
    \resizebox{2in}{!}{\includegraphics{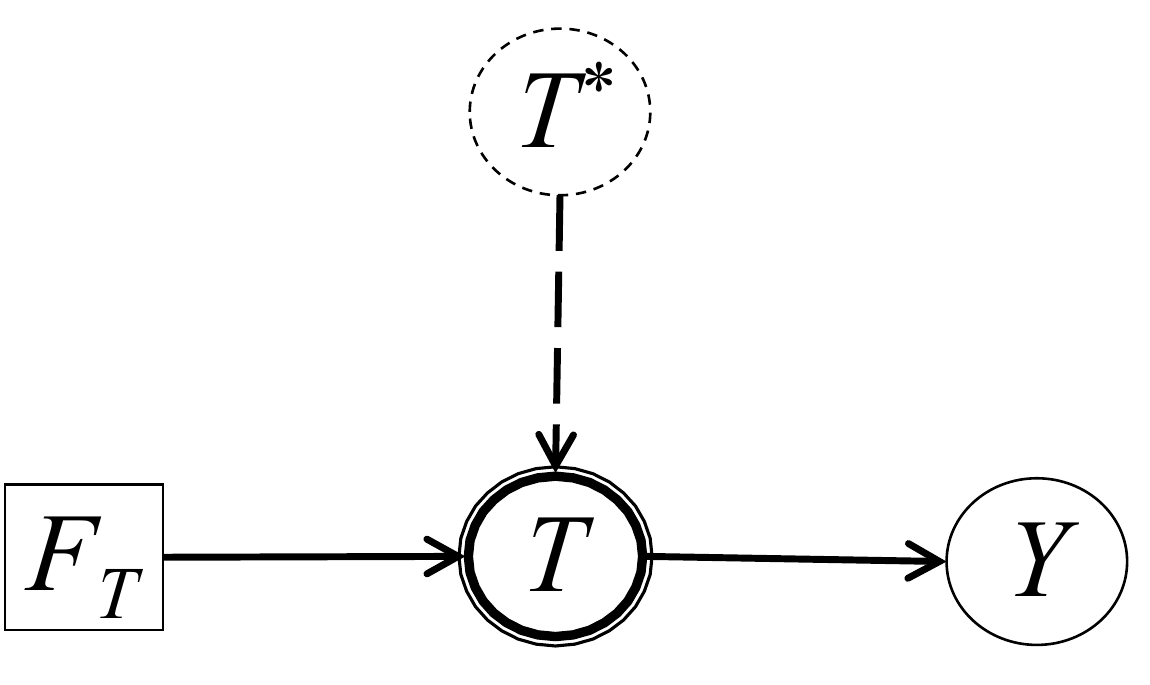}}
    \caption{Modification of \figref{basic} representing ignorability}
    \label{fig:igndag}
  \end{center}
\end{figure}

\begin{rem}
  \label{rem:altgraph}
  We might try and make the deletion of the arrow from $T^*$ to $Y$ in
  \figref{basic} into a graphically based argument for
  \lemref{ignalg}, for it appears to impose just the additional
  conditional independence property \eqref{ignf} representing
  ignorability, and to imply the desired result \eqref{igndag}.
  However this is again a misleading argument: inference from such
  surgery on a DAG can only be justified when it has a basis in the
  algebraic theory of conditional independence
  \cite{apd:CIST,pc/apd:eci}, which here it does not, on account of
  the fallacious argument identified in \remref{altalg}.
\end{rem}

\figref{collign} results on ``eliminating $T^*$'' from \figref{igndag}:
that is to say, the conditional independencies represented in
\figref{collign} are exactly those of \figref{igndag} that do not
involve $T^*$.  In this case, the only such  property is
\eqref{igndag}.
\begin{figure}[htbp]
  \begin{center}
    \resizebox{2in}{!}{\includegraphics{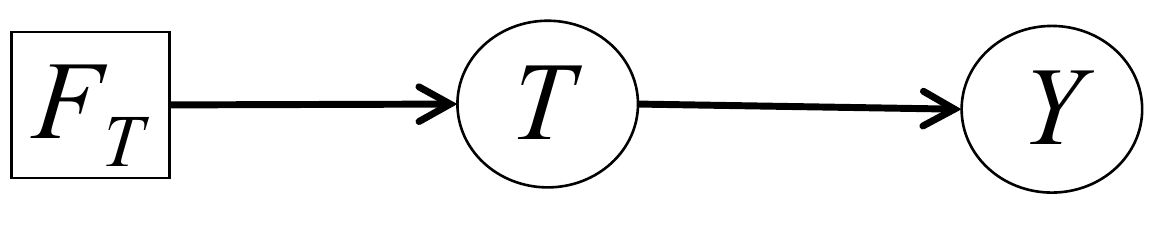}}
    \caption{Collapsed DAG under ignorability, representing
      $\ind Y {F_T} T$}
    \label{fig:collign}
  \end{center}
\end{figure}

The ECI property \eqref{igndag}, and the DAG of \figref{collign}, are
the basic (respectively algebraic and graphical) representations of
``no confounding'' in the DT approach, which has been treated as a
primitive in earlier work.  The above analysis supplies deeper
understanding of these representations.  Although on getting to this
point we have been able to eliminate explicit consideration of the
treatment selection variable $T^*$, our more detailed analysis, which
takes it into account, makes clear just what needs to be argued in
order to justify \eqref{igndag}: namely, the property of ignorability
expressed algebraically by \eqref{indyft} or \eqref{ignf} and
graphically by \figref{igndag}, and further described in
\secref{ignore}.

\section{Covariates}
\label{sec:cov}

The ignorability assumption \eqref{indys_t} will often be untenable.
If, for example, those fingered for treatment (so with $T^*=1$) are
sicker than those fingered for control ($T^*=0$)---as might well be
the case in a non-randomised study---then (under either treatment
application $\check T = t$, $t =0,1$) we would expect a worse outcome
$Y$ when knowing $T^* =1$ than when knowing $T^*=0$.  However, we
might be able to reinstate \eqref{indys_t} after further conditioning
on a suitable variable $X$ measuring how sick an individual is.  That
is, we might be able to make a case that, {\em after restricting
  attention to those individuals having a specified degree $X=x$ of
  sickness\/}, the further information that an individual had been
fingered for treatment would make no difference to the assessment of
the individual's response (under either treatment application).  This
would of course require that, after taking sickness into account, the
treatment assignment process was not further related to other possible
indicators of outcome (\eg, sex, age,\ldots).  If it is, these would
need to be included as components of the (typically multivariate)
variable $X$.  We assume that the appropriate variable $X$ is (in
principle at least) fully measurable, both for the individuals in the
study and (unlike $T^*$) for myself.  We assume internal
exchangeability of $(X, T^*, Y)$, extending this to external
exchangeability for $(X,Y)$.\footnote{This last condition could be
  relaxed, allowing my own distribution for $X$ to differ from that in
  the data, while retaining conditional exchangeability for $Y$, given
  $X$.  For simplicity we do not consider this further here.}

If and when such a variable $X$ can be identified, we will be able to
justify an assumption of {\em conditional ignorability\/}:
\begin{equation}
  \label{eq:condig}
  \ind Y {T^*} {(X,\check T)}.
\end{equation}

Furthermore, to be of any use in addressing my own decision problem,
such a variable must be a covariate, available prior to treatment
application, and so, in particular must (jointly with $T^*$, at least
for the study individuals, for whom $T^*$ is defined) have the same
distribution under either hypothetical treatment application.  This is
expressed as
\begin{equation}
  \label{eq:samedist}
  \indo {(X,T^*)} {\check T}.
\end{equation}
In particular, there will be a common marginal distribution, $P_X$
say, for $X$, in both interventional regimes.

When both \eqref{condig} and \eqref{samedist} are satisfied, we call
$X$ a {\em sufficient covariate\/}.  These properties are represented
by the DAG of \figref{suffcov0}.

\begin{figure}[htbp]
  \begin{center}
    \resizebox{2in}{!}{\includegraphics{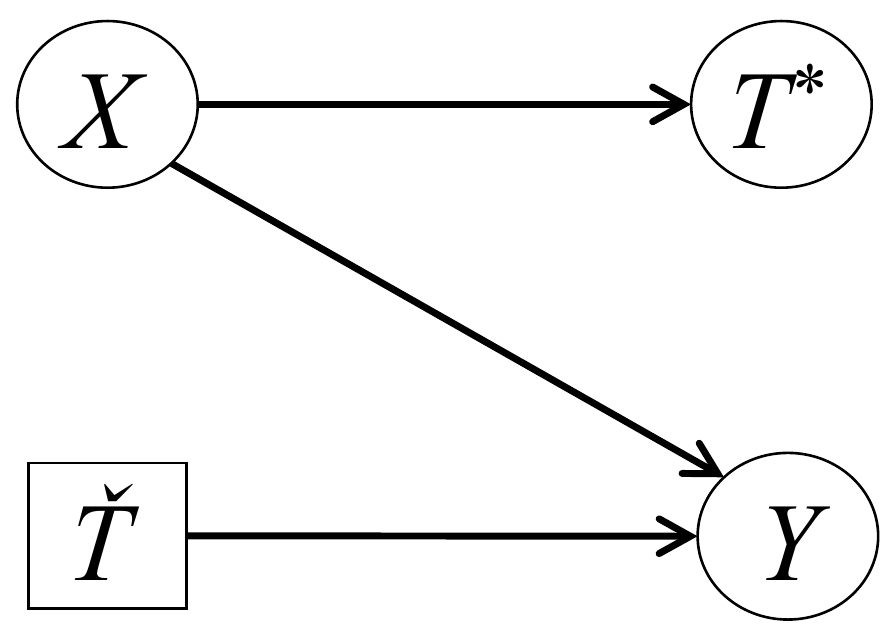}}
    \caption{DAG representing sufficient covariate $X$:
      $\indo {(X,T^*)} {\check T}$ and $\ind Y {T^*} {(X,\check T)}$.}
    \label{fig:suffcov0}
  \end{center}
\end{figure}

\subsection{Idle regime}
\label{sec:idlecov}
As in \secref{idle}, we introduce the regime indicator $F_T$, allowing
for consideration of the ``idle'' observational regime $F_T = \idle$,
in addition to the interventional regimes $F_T=t$ ($t=0,1)$; and the
constructed ``applied treatment'' variable $T$ of \defref{appl}.
Arguing as for \eqref{ignf}, \eqref{condig} implies
\begin{equation}
  \label{eq:xx}
  \ind Y {T^*} {(X,T,F_T)}.
\end{equation}

\begin{lem}
  \label{lem:covprops}
  Let $X$ be a sufficient covariate.  Then
  \begin{eqnarray}
    \label{eq:cov1}
    {(X,T^*)} &\cip&  {F_T}\\
    \label{eq:YXdist}
    Y &\cip& (T^*,F_T) \,\cd\, (X,T).
  \end{eqnarray}
\end{lem}

\begin{proof}
  By distributional consistency \eqref{idledist},
  \begin{eqnarray*}
    X \cd T^* =1, F_T = \idle &\approx& X \cd T^*=1, F_T = 1\\
                              &\approx& X \cd T^*=1, F_T = 0
  \end{eqnarray*} 
  by \eqref{samedist}.  Hence $\ind X {F_T} {T^*=1}$.  A parallel
  argument shows $\ind X {F_T} {T^*=0}$, so that $\ind X {F_T} {T^*}$.
  On combining this with \eqref{tstarsame} we obtain \eqref{cov1}.

  As for \eqref{YXdist}, this is equivalent to the conjunction of
  \eqref{xx} and $\ind Y {F_T} {(T,X)}$.  The argument for the latter
  (again, requiring distributional consistency) parallels that for
  \eqref{igndag}, after further conditioning on $X$ throughout.
\end{proof}

The properties \eqref{cov1} and \eqref{YXdist} are embodied in the DAG
of \figref{covstar}.
\begin{figure}[h]
  \begin{center}
    \resizebox{2in}{!}{\includegraphics{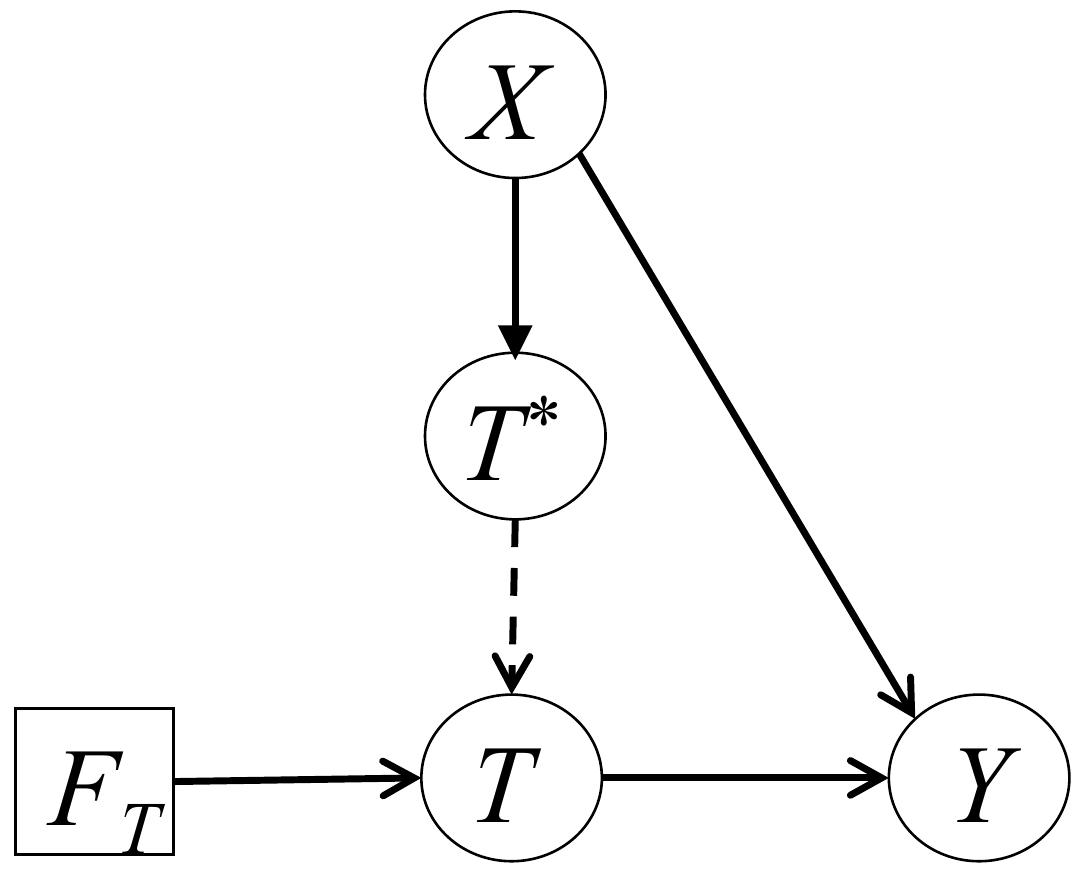}}
    \caption{Full DAG with sufficient covariate $X$ and regime indicator}
    \label{fig:covstar}
  \end{center}
\end{figure}
This implies, on eliminating the unobserved variable $T^*$:
\begin{eqnarray}
    \label{eq:cov2}
    {X} &\cip&  {F_T}\\
    \label{eq:YXdist2}
    Y &\cip& F_T \,\cd\, (X,T),
  \end{eqnarray}
as represented by \figref{covnostar}.
\begin{figure}[h]
  \begin{center}
    \resizebox{2in}{!}{\includegraphics{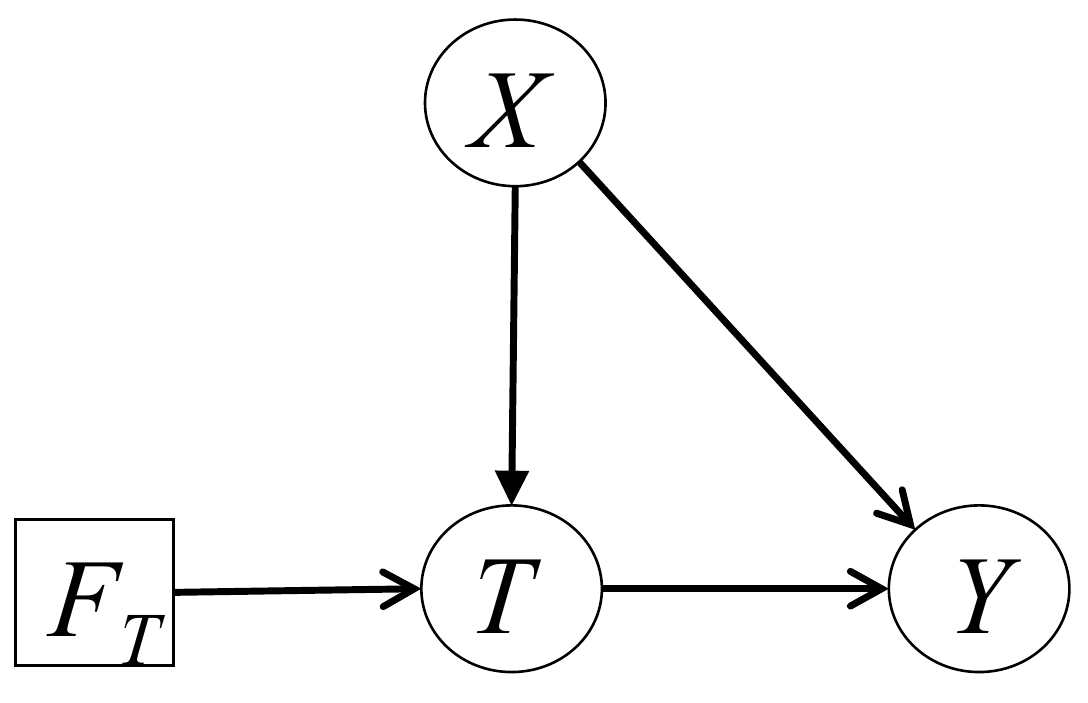}}
    \caption{Reduced DAG with sufficient covariate $X$ and regime indicator}
    \label{fig:covnostar}
  \end{center}
\end{figure}
Properties \eqref{cov2} and \eqref{YXdist2}, as embodied in
\figref{covnostar}, are the basic DT representations of a sufficient
covariate.  Assuming $X$, $T$ and $Y$ are all observed, this is what
is commonly referred to as ``no unmeasured confounding''.

\section{More complex DAG models}
\label{sec:complex}

\subsection{An example}
\label{sec:dagex}

Consider the following story.  In an observational setting, variable
$X_0$ represents the initial treatment received by a patient; this is
supposed to be applied independently of an (unobserved) characteristic
$H$ of the patient.  The variable $Z$ is an observed response
depending, probabilistically, on both the applied treatment $X_0$ and
the patient characteristic $H$. A subsequent treatment, $X_1$, can
depend probabilistically on both $Z$ and $H$, but not further on
$X_0$.  Finally the distribution of the response $Y$, given all other
variables, depends only on $X_1$ and $Z$.  \figref{pearlswiga} is a
DAG representing this story by means of $d$-separation.

\begin{figure}[htbp]
  \centering
  \includegraphics[width=.4\textwidth]{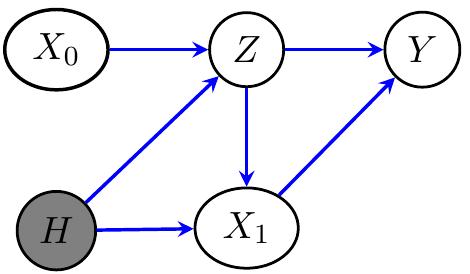}
  \caption{Observational DAG}
  \label{fig:pearlswiga}
\end{figure}

In addition to the observational regime, we want to consider possible
interventions to set values for $X_0$ and $X_1$.  We thus have two
non-stochastic regime indicators, $F_0$ and $F_1$: $F_i = x_i$
indicates that $X_i$ is externally set to $x_i$, while
$F_i = \emptyset$ allows $X_i$ to develop ``naturally''.  The overall
regime is thus determined by the pair $(F_0,F_1)$.

\figref{pearlstara} augments \figref{pearlswiga}, in a seemingly
natural way, to include these regime indicators.  It represents, by
$d$-separation, ways in which the domain variables are supposed to
respond to interventions.
\begin{figure}[htbp]  \centering
  \includegraphics[width=.4\textwidth]{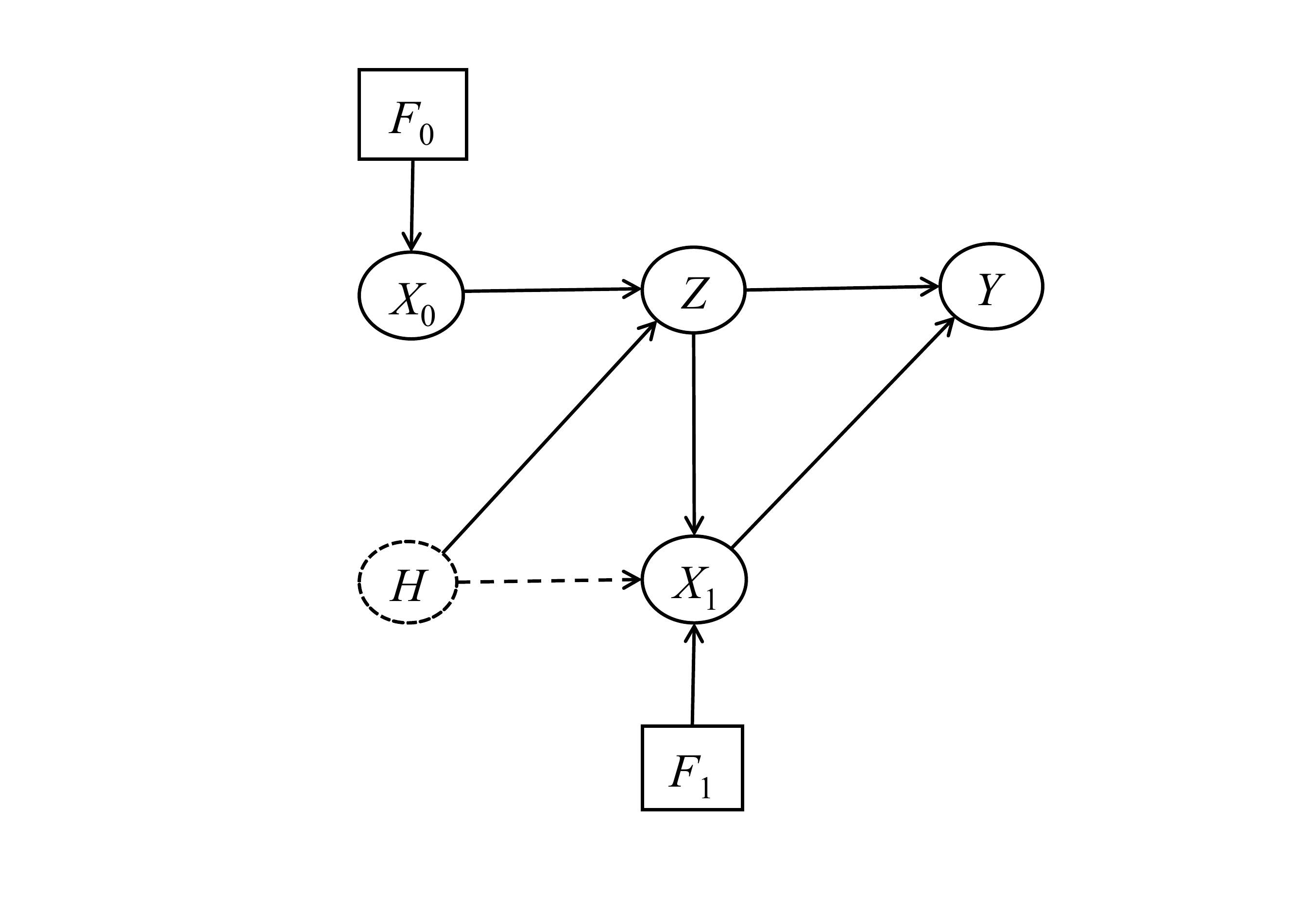}
  \caption{Augmented DAG}
  \label{fig:pearlstara}
\end{figure}
For example, it implies $\ind Y {(X_0, H, F_0, F_1)} {(Z.X_1)}$: once
we know $Z$ and $X_1$, not only are $X_0$ and $H$ irrelevant for
probabilistic prediction of $Y$, but so too is the information as to
whether either or both of $X_0$, $X_1$ arose naturally, or were set by
intervention.  In particular, the conditional distribution of $Y$
given $(Z,X_1)$, under intervention at $X_1$, is supposed the same as
in the observational regime modelled by \figref{pearlswiga}.

\subsubsection{From observational to augmented DAG }
\label{sec:conver}

It does not follow, merely from the fact that we can model the
observational conditional independencies between the domain variables
by \figref{pearlswiga}, that their behaviour under the entirely
different circumstance of intervention must be as modelled by
\figref{pearlstara}.  Strong additional assumptions are required to
bridge this logical gap.  These we now elaborate.

We again introduce ``intention to treat'' variables, $X_0^*$ and
$X_1^*$, the realised $X_0$ and $X_1$, in any regime, being given by
\begin{equation}
  \label{eq:daginter}
  X_i = \left\{
    \begin{array}[c]{ll}
      X^*_i   &\mbox{if $F_i = \emptyset$}\\
      F_i &  \mbox{if $F_i\neq \emptyset$}.
    \end{array}
  \right.
\end{equation}
Since, in the observational regime, $X_i=X_i^*$, \figref{pearlswiga}
would still be observationally valid on replacing each $X_i$ by
$X_i^*$.

The different regimes are supposed linked together by the following
assumptions, which we first present and then motivate:
\begin{eqnarray}
  \label{eq:dagign0}
  {X_0^*} &\cip& {(F_0,F_1)}\\
  \label{eq:dagign1a}
  {(H,Z,X_1^*,Y)} &\cip& {(F_0,X_0^*)} \mid {(F_1,X_0)}\\
  \label{eq:dagign2a}
  (X_0^*,H,Z,X_1^*) &\cip& F_1 \mid F_0\\
  \label{eq:dagign3}
  Y &\cip& (F_1, X_1^*) \mid (F_0,X_0,H,Z,X_1).
\end{eqnarray}

Note that, since $X_i$ is determined by $(F_i,X_i^*)$, \eqref{dagign1a}
and \eqref{dagign2a} are equivalent to:
\begin{eqnarray}
  \label{eq:dagign1}
  {(H,Z,X_1^*,X_1,Y)} &\cip& {(F_0,X_0^*)} \mid {(F_1,X_0)}\\
  \label{eq:dagign2}
  (X_0^*,X_0,H,Z,X_1^*) &\cip& F_1 \mid F_0.
\end{eqnarray}

\paragraph{Comments on the assumptions}

\renewcommand{\theenumi}{(\roman{enumi})} In order to understand the
above assumptions, we should consider \figref{pearlswiga} as
describing, not only the conditional independencies between variables,
but also a partial order in which the variables are generated: it is
supposed that, in any regime, the value of a parent variable is
determined before that of its child.  In particular it is assumed that
an intervention on a variable can not affect that variable's
non-descendants---including their intention-to-treat variables and its
own; but may affect its descendants---including their associated
intention-to-treat variables.

\begin{enumerate}
\item \label{it:x0} Similar to \eqref{tstarsame}, \eqref{dagign0}
  expresses the property that an intention-to-treat variable, here
  $X_0^*$, should behave as a covariate for $X_0$, and so be
  independent of which regime, here $F_0$, is operating on $X_0$.
  Moreover, $X_0^*$ should not be affected by a subsequent
  intervention (or none), $F_1$, at $X_1$.
  
\item \label{it:x0a} Assumption~\eqref{dagign1a} is a version of the
  ignorability property \eqref{ddd}.  It says that an intervention on
  $X_0$ should be ignorable in its effect on all other variables.
  Moreover this should apply conditional on $F_1$, \ie\ whether or not
  there is an intervention at $X_1$.
  \begin{rem}
    \label{rem:withdc}
    As previously discussed, ignorability is a strong assumption,
    requiring strong justification.  Also note that, as shown by
    \corref{aaa}, \eqref{dagign1a} is implicitly assuming the
    distributional consistency property (\defref{distcons}), in
    addition to ignorability.
  \end{rem}
 
\item \label{it:x1} Assumption~\eqref{dagign2a} expresses the
  requirement that $(X_0^*,H,Z,X_1^*)$, being generated prior to
  $X_1$, should not be affected by intervention $F_1$ at $X_1$.
  (However, they might depend on which regime, $F_0$, operates on
  $X_0$.)
 
\item \label{it:x1a} Similar to \itref{x0a}, \eqref{dagign3} says
  that, conditional on all the domain variables, $(X_0, H, Z)$,
  generated prior to $X_1$, the effect of intervention $F_1$ at $X_1$
  is ignorable for its effect on $Y$; moreover, this should hold
  whether or not there is intervention $F_0$ at $X_0$.  Informally,
  taken together with \eqref{dagign2}, this requires that $(X_0,H,Z)$
  form a sufficient covariate for the effect of $X_1$ on $Y$.
\end{enumerate}

In the following we make extensive (but largely implicit) use of the
axiomatic properties of (extended) conditional independence
\cite{apd:CIST,pearl:88}:
 \begin{description}
  \item[P\boldmath$1$ (Symmetry):] \ind{X}{Y}{Z} $\Rightarrow$
    \ind{Y}{X}{Z}.
  \item[P\boldmath$2$:] \ind{X}{Y}{Y}.
  \item[P\boldmath$3$ (Decomposition):] \ind{X}{Y}{Z} {\rm and} $W$ a
    function of $Y$ $\Rightarrow$ \ind{X}{W}{Z}.
  \item[P\boldmath$4$ (Weak Union):] \ind{X}{Y}{Z} {\rm and} $W$ a
    function of $Y$ $\Rightarrow$ \ind{X}{Y}{(W,Z)}.
  \item[P\boldmath$5$ (Contraction):] \ind{X}{Y}{Z} {\rm and}
    \ind{X}{W}{(Y,Z)} $\Rightarrow$ \ind{X}{(Y,W)}{Z}.
  \end{description}

\begin{lemma}
  \label{lem:stardag}
  Suppose that the observational conditional independencies are
  represented by \figref{pearlswiga}, and that assumptions
  \eqref{dagign0}--\eqref{dagign3} apply.  Then the extended
  conditional independencies between domain variables,
  intention-to-treat variables and regime indicators are represented
  by \figref{pearlstarb}.

  \begin{figure}[htbp]
    \begin{center}
      \includegraphics[width=.4\textwidth]{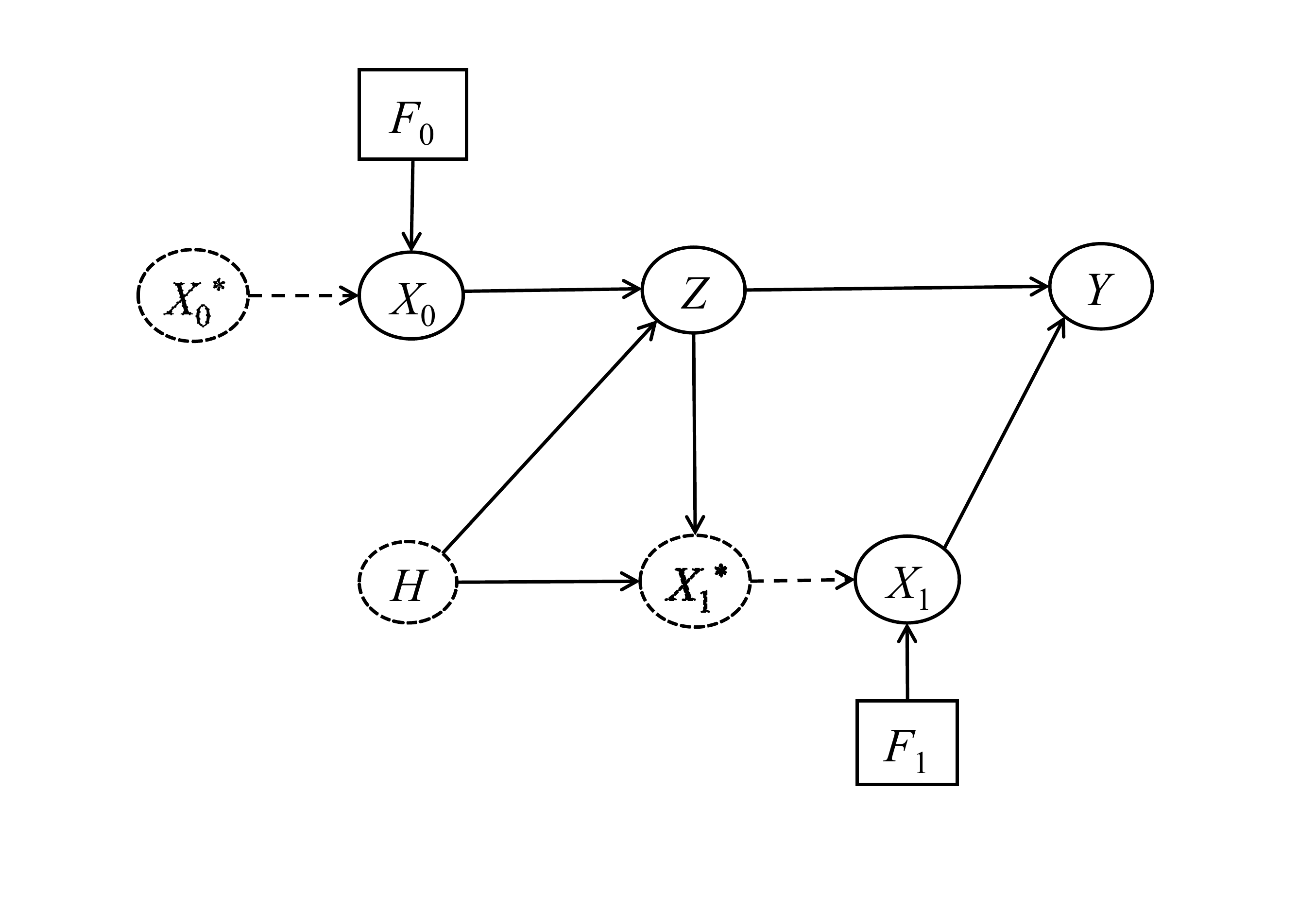}
    \caption{ITT DAG}
    \label{fig:pearlstarb}
  \end{center}
\end{figure}
\end{lemma}

\begin{rem}
  \label{rem:varind}
  A further property apparently represented in \figref{pearlstarb} is
  the independence of $F_0$ and $F_1$:
  \begin{equation}
    \label{eq:varind}
    \indo {F_0} {F_1}.  
  \end{equation}
  Now so far we have been able to meaningfully interpret an extended
  conditional independence assertion only when the left-hand term
  involves stochastic variables only---which seems to render
  \eqref{varind} meaningless.  Nevertheless, as a purely instrumental
  device, it is helpful to extend our understanding by considering the
  regime indicators as random variables also.  So long as all our
  assumptions and conclusions do not involve regime indicators in
  their left-hand term, any proof that uses this extended
  understanding will remain valid for the actual case of
  non-stochastic regime variables, as may be seen by conditioning on
  these.\footnote{It is not strictly necessary to regard the regime
    indicators as stochastic.  Instead we can interpret \eqref{varind}
    as expressing the non-stochastic property of {\em variation
      independence\/} \cite{apd:varind}, meaning that the range of
    possible values for each is unconstrained by the value taken by
    the other.  Indeed, this is implicit in our interpretive comments
    \itref{x0}--\itref{x1a} on conditions
    \eqref{dagign0}--\eqref{dagign3}.  We can then combine these two
    distinct interpretations of independence within the same
    application, as we do here.  For a rigorous analysis see
    \textcite{pc/apd:eci}.}
\end{rem}

In the light of \remref{varind}, we shall in the sequel treat $F_0$
and $F_1$ as stochastic variables, having the independence property
\eqref{varind}.\\

\begin{proof0}{\lemref{stardag}}  
  Taking the variables in the order
  $F_0, F_1, X_0^*,X_0,H,Z,X_1^*,X_1,Y$, we need to show the following
  series of properties, where each asserts the independence of a
  variable from its predecessors, conditional on its parents in the
  graph.
  \begin{eqnarray}    
    \label{eq:s0}
    {F_1} &\cip& {F_0}\\
    \label{eq:s1}
    X_0^* &\cip& (F_0,F_1)\\
    \label{eq:s2}
    X_0  &\cip& F_1 \mid (X_0^*, F_0)\\
    \label{eq:s3}
    H  &\cip& (F_0, F_1,X_0^*,X_0)\\
    \label{eq:s4}
    Z  &\cip& (F_0, F_1, X_0^*) \mid (X_0, H)\\
    \label{eq:s5}
    X_1^* &\cip& (F_0, F_1, X_0^*, X_0) \mid (H,Z)\\
    \label{eq:s6}
    X_1  &\cip& (F_0, X_0^*, X_0, H,Z) \mid (X_1^*, F_1)\\
    \label{eq:s7}
    Y  &\cip&  (F_0,F_1, X_0^*, X_0, H,X_1^*) \mid (Z,X_1).
  \end{eqnarray}
  On excluding \eqref{s0}, these conclusions will comprise the desired
  result.
  
  \begin{description}
  \item[For \eqref{s0}:] By assumption \eqref{varind}.
  \item[For \eqref{s1}:] By \eqref{dagign0}.
  \item[For \eqref{s2}:] Follows trivially since $X_0$, being
    functionally determined by $(X_0^*, F_0)$, has a conditional
    one-point distribution, and so is independent of anything else.
  \item[For \eqref{s3}--\eqref{s5}:] 
    From \eqref{dagign1} we have
    \begin{equation}
      \label{eq:pp1}
      \ind {(H,Z,X_1^*)} {F_0} {(F_1,X_0)}
    \end{equation}
    while from  \eqref{dagign2} we have
    \begin{equation}
      \label{eq:pp2}
      \ind {(H,Z,X_1^*)} {F_1} {(F_0,X_0)}.
    \end{equation}
    We now wish to show that \eqref{pp1} and \eqref{pp2} imply
    \begin{equation}
      \label{eq:pp3}
      \ind {(H,Z,X_1^*)} {(F_0,F_1)} {X_0}.
        \end{equation}
        This requires some caution, on account of \remref{altalg}.  To
        proceed we use the fictitious independence property
        \eqref{varind}.

        From \eqref{dagign2} we have $\ind {X_0} {F_1} {F_0}$, which
        together with \eqref{varind} yields $\indo {F_1} {(F_0,X_0)}$,
        so that
        \begin{equation}
          \label{eq:fx}
          \ind  {F_1} {F_0} {X_0}.
        \end{equation}
        Combining \eqref{pp1} and \eqref{fx} yields
        \begin{math}
          \ind {(F_1,H,Z,X_1^*)} {F_0} {X_0}
        \end{math}
        whence
        \begin{equation}
          \label{eq:fw}
           \ind {(H,Z,X_1^*)} {F_0} {X_0}.
        \end{equation}
        Finally, combining \eqref{fw} and \eqref{pp2} yields
        \eqref{pp3}.
    
        Now \eqref{pp3} asserts that the conditional distribution of
        $(H,Z,X_1^*)$ given $X_0$ is the same in all regimes.  In
        particular (noting that $X_1^*=X_1$ in the observational
        regime), that conditional distribution inherits the
        independencies of \figref{pearlswiga}.  Properties
        \eqref{s3}--\eqref{s5} follow (on noting that $X_0$, being a
        function of $F_0$ and $X_0^*$, is redundant in \eqref{s3} and
        \eqref{s5}).

  \item[For \eqref{s6}:]  Trivial since $X_1$ is
    functionally determined by $(F_1,X_1^*)$.
  \item[For \eqref{s7}:] From \eqref{dagign1} we derive both
    \begin{eqnarray}
      \label{eq:p16}
       Y &\cip& F_0 \mid (F_1,X_0,H,Z,X_1)\\
        \label{eq:p17}
       Y &\cip& X_0^* \mid (F_0,F_1, X_0, H,Z,X_1^*,X_1).             
     \end{eqnarray}
     while from \eqref{dagign3} we have
     \begin{eqnarray}
       \label{eq:p18}
       Y &\cip& F_1 \mid (F_0,X_0,H,Z,X_1)\\
        \label{eq:p18a}
       Y &\cip& X_1^* \mid (F_0,F_1,X_0,H,Z,X_1).
     \end{eqnarray}
     
    We first want to show that \eqref{p16} and \eqref{p18} are
    together equivalent to
     \begin{equation}
       \label{eq:p19}
      \ind Y  {(F_0,F_1)}  {(X_0,H,Z,X_1)}.
    \end{equation}
    To work towards this, we note that, by \eqref{dagign1},
    \begin{math}
      \ind {(H,Z,X_1)} {F_0} {(F_1,X_0)},
    \end{math}
    which together with \eqref{fx} gives
    \begin{math}
      \ind {(F_1,H,Z,X_1)} {F_0} {X_0},
    \end{math}
    whence
    \begin{equation}
      \label{eq:ff}
      \ind {F_0} {F_1}  {(X_0,H,Z,X_1)}.
    \end{equation}
    Then \eqref{p19} follows from \eqref{p16}, \eqref{p18} and
    \eqref{ff} in parallel to the argument above from \eqref{pp1},
    \eqref{pp2} and \eqref{fx} to \eqref{pp3}.
    
    Now in the observational regime, $\ind Y {(X_0,H)}{(Z,X_1)}$.  By
    \eqref{p19}, this must hold in all regimes.  This gives
    \begin{equation}
       \label{eq:p20}
       \ind Y  {(F_0,F_1,X_0,H)} {(Z,X_1)}.
     \end{equation}
     Properties \eqref{p18a} and \eqref{p20} are together equivalent
     to
     \begin{equation}
       \label{eq:p21}
       \ind Y  {(F_0,F_1,X_0,X_1^*,H)}  {(Z,X_1)}.
     \end{equation}
     Combining \eqref{p21} with \eqref{p17} now yields \eqref{s7}.
  \end{description}
\end{proof0}

\paragraph{Augmented DAG} Finally, having derived \figref{pearlstarb}
from assumptions \eqref{dagign0}--\eqref{dagign3}, we can eliminate
$X_0^*$ and $X_1^*$ from it.  The relationships between the domain and
regime variables are then represented by the augmented DAG of
\figref{pearlstara}, which can now be used to express and manipulate
causal properties of the system, without further explicit
consideration of the ITT variables---such consideration only being
required in making the argument to justify this use.

\subsection{General DAG}
\label{sec:daggen}
The case of a general DAG follows by extension of the arguments of
\secref{dagex} above.  Consider a set of domain variables, with
observational independencies represented by a DAG $\dag$.  We consider
the variables in some total ordering consistent with the partial order
of the DAG.

Some of the variables, say (in order) $(X_i:i=1,\ldots,k)$, will be
potential targets for intervention, with associated intention-to-treat
variables $(X_i^*)$ and intervention indicator variables ($F_i$).  Let
$V_i$ denote the set of all the domain variables coming between
$X_{i-1}$ and $X_i$ in the order.  We thus have an ordered list
$L = (V_1,X_1,\ldots, V_k, X_k, V_{k+1})$ of domain
variables, some of which are possible targets for intervention.

Let $\pre_i$ denote the set of all predecessors of $X_i$ in $L$,
including $X_i$, and $\suc_i$ the set of all successors of $X_i$,
excluding $X_i$.  By $\pre_i^*$ we understand the set where all action
variables in $\pre_i$ are replaced by their associated
intention-to-treat variables, and similarly for $\suc_i^*$.  Also
$F_{i:j}$ will denote $(F_i,\ldots,F_j)$, and similarly for other
variables%
.

Generalising \eqref{dagign0} with \eqref{dagign1a}, or
\eqref{dagign2a} with \eqref{dagign3}, and with similar motivation, we
introduce the following assumptions (noting that $B_i$ expresses a
strong ignorability property for the effects of all the variables
$(X_1,\ldots,X_i)$ on later variables---which would need
correspondingly strong justification in any specific application):

\begin{eqnarray}
  \label{eq:Ai}
  A_i&:&\qquad \pre_i^*\, \cip \, F_{i:k} \mid F_{1:i-1}\\
  \label{eq:Bi}
  B_i&:&\qquad \suc_i^* \,\cip\, (F_{1:i}, X_{1:i}^*) \mid (F_{i+1:k},\,\pre_i).
\end{eqnarray}

Taking account of the fact that $X_i$ is determined by $(F_i,X_i^*)$,
these are equivalent to:

\begin{eqnarray}
  \label{eq:Ai'}
  A_i'&:&\qquad (V_{1:i},X_{1:i}^*,X_{1:i-1})\, \cip \, F_{i:k} \mid F_{1:i-1}\\
  \label{eq:Bi'}
  B_i'&:&\qquad (V_{i+1:k},X_{i+1:k}^*,X_{i+1:k}) \,\cip\,  (F_{1:i}, X_{1:i}^*)
         \mid (F_{i+1:k},V_{1:i},X_{1:i}).
\end{eqnarray}

\begin{theorem}
  \label{thm:stardaggen}
  Suppose the observational conditional independencies are represented
  by a DAG $\dag$, and that assumptions $A_i$ and $B_i$
  ($i=1,\ldots,k$) hold.  Then the extended conditional independencies
  between domain variables, intention-to-treat variables, and regime
  variables (conditional on the regime variables) are represented by
  the ITT DAG $\dag^*$, constructed by modifying $\dag$ as follows:

  \begin{itemize}
  \item Each action variable $X_i$ is replaced by the trio of
    variables $F_i$, $X_i^*$ and $X_i$, with arrows from $F_i$ and
    $X_i^*$ to $X_i$.  It is assumed that \eqref{daginter} holds.
  \item $F_i$ is a founder node.
  \item $X_i^*$ inherits all the original incoming arrows of $X_i$.
  \item $X_i$ loses its original incoming arrows, but retains its
    original outgoing arrows.
  \end{itemize}
\end{theorem}

\begin{proof}
  See \appref{thmproof}.
\end{proof}

Finally, on eliminating the intention-to-treat nodes $(X_i^*)$ from
the ITT DAG, the relationships between the domain variables and regime
variables are represented by the augmented DAG $\dag^\dagger$,
constructed from $\dag$ by adding, for each $X_i$, $F_i$ as a founder
node, with an arrow from $F_i$ to $X_i$.  As described in \secref{dt},
such an augmented DAG is all we need to represent and manipulate
causal properties.  The above argument shows what needs to be
assumed---and, more important, justified---to validate its use.

\section{Comparison with other approaches}
\label{sec:comparison}
In this section we explore some of the similarities and differences
between the decision-theoretic approach to statistical causality,
considered above, and other currently popular approaches.

\subsection{Potential outcomes}
\label{sec:po}
In the potential outcome (PO) formulation of statistical causality
\cite{dbr:jep,dbr:as}, the conception is that (for a generic
individual) there exist, simultaneously and before the application of
any treatment, two variables, $Y(0)$ and $Y(1)$: $Y(t)$ represents the
individuals's {\em potential response\/} to the (actual or
hypothetical) application of treatment $t$.  If treatment $1$ (resp.,
$0$) is in fact applied, the corresonding potential outcome $Y(1)$
(resp., $Y(0)$) will be uncovered and so rendered actual, the observed
response then being $Y = Y(1)$ (resp., $Y = Y(0)$); however the
alternative, counterfactual, potential outcome $Y(0)$ (resp., $Y(1)$)
will remain forever unobserved---a feature which \textcite{pwh:jasa}
has termed the {\em fundamental problem of causal inference\/},
although it is not truly fundamental, but rather an artefact of the
unnecessarily complicated PO approach.

The pair $(Y(1), Y(0))$ is supposed to have (jointly with the other
variables in the problem) a bivariate distribution, common for all
individuals---this might be regarded as generated from an assumption
of exchangeability of the pairs $(Y_i(1),Y_i(0))$ across all
individuals $i\in {\cal I}$.  The marginal distribution of $Y(t)$ can
be identified with our hypothetical distribution $P_t$ for the
(single) response variable $Y$ under hypothesised application of
treatment $t$, and is thus estimable from suitable experimental data.
However, on account of the fundamental problem of causal inference no
empirical information is obtainable about the dependence between
$Y(0)$ and $Y(1)$, which can never be simultaneously observed.

\paragraph{Causal effect}
If I (individual $0$) consider taking treatment $1$ [resp., $0$], I
would then be looking forward to obtaining response $Y_0(1)$ [resp.,
$Y_0(0)$].  Causal interest, and inference, will thus centre on a
suitable comparison between the two potential responses.  The PO
approach typically focuses on the ``individual causal effect'',
$\ice := Y(1)-Y(0)$.  However, again on account of the fundamental
problem of causal inference, $\ice$ is never directly observable, and
even its distribution can not be estimated from data except by making
arbitrary and untestable assumptions (\eg, that $Y(1)$ and $Y(0)$ are
independent, or alternatively---''treatment-unit additivity,
TUA''---that they differ by a non-random constant).  For this reason
attention is typically diverted to the {\em average causal effect\/},
$\ace := \E(\ice)$.  Since this can be re-expressed as
$\E \{Y(1)\}-\E \{Y(0)\}$, and the individual expectations are
estimable, so is $\ace$: indeed, although based on a different
interpretation and expressed in different notation, it is essentially
the same as our own definition \eqref{ace} of $\ace$, which was
introduced as one form of comparison between the two {\em
  distributions\/}, $P_1$ and $P_0$, for the single response
$Y$---rather than, as in the PO approach, an estimable distributional
feature of the non-estimable comparison $\ice$ between the two {\em
  variables\/} $Y(1)$ and $Y(0)$.

\paragraph{Consistency}
In the PO approach, {\em consistency\/} refers to the property
\begin{equation}
\label{eq:cons}
Y = Y(T),
\end{equation}
requiring that the response $Y$ should be obtainable by revealing the
potential response corresponding to the received treatment $T$.  We
can distinguish two aspects to this:
\begin{enumerate}
\item \label{it:cons1}When considered only in the context of an
  interventional regime $F_T=t$, \eqref{cons} can be regarded as
  essentially a book-keeping device, since $Y(t)$ is {\em defined\/}
  as what would be observed if treatment $t$ were applied.
\item \label{it:cons2}But when it is understood as applying also in
  the observational regime, \eqref{cons} has more bite, requiring that
  an individual's response to received treatment $T$ should not depend
  on whether that treatment was applied by a (real or hypothetical)
  extraneous intervention, or, in the observational setting, by some
  unknown internal process.  It is thus a not entirely trivial
  modularity assumption, forming the essential link between the
  observational and interventional regimes.
\end{enumerate}

A parallel to aspect \itref{cons1} in DT is the {\em temporal
  coherence\/} assumption appearing in \fnref{coherence}: this
requires that uncertainty about the outcome $Y$, after it is known
that treatment $t$ has been applied, should be the same as the initial
uncertainty about $Y$, on the hypothesis that treatment $t$ will be
applied.  While not entirely vacuous, this too could be considered as
little more than book-keeping.

More closely aligned with aspect \itref{cons2} is the distributional
consistency property expressed in \eqref{idledist}, which says that,
for purposes of assessing the uncertainty about the response to a
treatment $t$, the {\em only\/} difference between the interventional
and the observational regime is that, in the latter, we have the
additional information that the individual had been fingered to
receive $t$.  Again this has some empirical bite, and can be regarded
as a not entirely trivial condition linking the observational and
interventional regimes in the DT approach.

\paragraph{Treatment assignment and application}
We have emphasised the distinction between the stochastic treatment
assignment variable $T^*$ and the non-stochastic treatment application
indicator $\check T$.  This is not explicitly done in the PO approach,
but appears implicitly, since for any data individual, with fingered
(and thus also actual) treatment $T^*$ (typically just denoted by $T$
in PO), we can distinguish between the actual response $Y = Y(T)$ in
the observational regime, and the potential responses $Y(1)$ and
$Y(0)$, relevant to the two interventional regimes.

We can make the following correspondences:
\begin{table}[h]
  \centering
  \begin{tabular}{rll}
    & PO & DT\\\hline
    (i) & Distribution of $Y(t)$  & Distribution of $Y$ given $\check T = t$\\
    (ii) & Joint  distribution of $(Y(0), Y(1))$  & no parallel\\ 
    (iii) & Distribution of $Y$ given $T = t$  & Distribution of $Y$ given $T^*=t, \check T = t$\\
    (iv) & $\indo {Y(t)} T$ \quad ($t=0,1$) & $\ind Y {T^*} {\check T}$\\
    (v) & $\indo {(Y(0), Y(1))} T$ & no parallel\\
    (vi) & $\ind {Y(t)} T X$ \quad ($t=0,1$) & $\ind Y {T^*} {(X, \check T)}$\\
    (vii) & $\ind {(Y(0), Y(1))} T X$ & no parallel  
  \end{tabular}
  \caption{Comparison of PO and DT approaches}
  \label{tab:comppo}
\end{table}

\paragraph{Ignorability}
The PO expressions in (iv) and (v) of \tabref{comppo} have both been
used to express ignorability in the PO framework, (iv) evidently being
weaker than (v).  The {\em weak ignorability\/} condition (iv)
corresponds directly to the DT condition \eqref{indys_t} for
ignorability.  However, the {\em strong ignorability\/} condition (v)
has no DT parallel, since nothing in DT corresponds to a joint
distribution of $(Y(0), Y(1))$.  For applications weak ignorability
(iv), which does have a DT interpretation, suffices.  Similar remarks
apply to the (weak and strong) {\em conditional ignorability\/}
expressions in (vi) and (vii).

\paragraph{SUTVA and SUTDA} It is common in PO to impose the {\em
  Stable Unit-Treatment Value Assumption (SUTVA)\/}
\cite{drb:basu,drb:ifs}.  This requires that, for any individual $i$,
the potential response $Y_i(t)$ to application of treatment $t$ to
that individual should be unaffected by the treatments applied to
other individuals.  Indeed, without such an assumption the notation
$Y_i(t)$ becomes meaningless, since the very concept intended by it is
denied.

Our variant of SUTVA is the {\em Stable Unit-Treatment Distribution
  Assumption (SUTDA)\/}, as described in \condref{sutda}.  (Note that,
unlike for SUTVA, even when this assumption fails it does not
degenerate into meaninglessness, since the terms in it have
interpretations independent of its truth).  On making the further
assumption, implicit in the PO approach, that, not just the set of
values, but also the joint distribution, of the collection
$\{Y_i(t): i\in{\cal I}, t\in{\cal T}\}$ is unaffected by the
application of treatments, it is easily seen that SUTVA implies SUTDA,
so that our condition is weaker---and is sufficient for causal
inference.

\subsection{Pearlian DAGs}
\label{sec:pearl}

Judea Pearl has popularised graphical representations of causal
systems based on DAGs.  In \S1.3 of \textcite{pearl:book} he describes
what he terms a ``Causal Bayesian Network'' (CBN), which we shall call
a ``Pearlian DAG''.\footnote{We avoid the term ``causal DAG'', which
  has been used with a variety of different interpretations
  \cite{apd:beware}.}  This is intended to represent both the
conditional independencies between variables in observational
circumstances, and how their joint distributions changes when
interventions are made on some or all of the variables: specifically,
for any node not directly intervened on, its conditional distribution
given its parents is supposed the same, no matter what other
interventions are made.\footnote{In the greater part of his causal
  writings, Pearl uses a different construction, in which all
  stochasticity is confined to unobservable ``error variables'', with
  domain variables related to these, and to each other, by
  deterministic functional relationships---he misleadingly terms this
  deterministic structure a ``probabilistic causal model'' (PCM).  It
  is easy to show \cite{apd:infdiags} that there is a many-one
  correspondence: any PCM implies a CBN structure for its domain
  variables, while any CBN can be derived from a, typically
  non-unique, PCM.  Since the additional, unidentifiable, structure
  embodied in a PCM has no consequences for its use for
  decision-theoretic purposes, we do not consider these further here.}
The semantics of a Pearlian DAG representation is in fact identical
with that, based entirely on \mbox{$d$-separation}, of the fully
augmented observational DAG, in which every observable domain variable
is accompanied by a regime indicator---thus allowing for the
possibility of intervention on every such variable.  However, although
Pearl has occasionally included these regime indicators explicitly, as
do we, for the most part he uses a representation where they are left
implicit and omitted from the graph.  A Pearlian DAG then looks,
confusingly, exactly like the observational DAG, with its conditional
independendies, but is intended to represent additional causal
properties: properties that are {\em explicitly\/} represented by the
corresponding augmented DAG.

Since such a Pearlian DAG is just an alternative representation of a
particular kind of augmented DAG, its appropriateness must once again
depend on the acceptability of the strong assumptions, described in
\secref{daggen}, needed to justify augmentation of an observational
DAG.

\subsection{SWIGs}
\label{sec:swigs}

\citet{Richardson_primer,Richardson_singleworld} introduce a different
graphical representation of causal problems, the SWIG (single world
intervention graph).  A salient feature of this approach is
``node-splitting'', whereby a variable is represented twice: once as
it appears naturally, and again as it responds to as intervention.
Although the details of their representation and ours differ, they are
based on similar considerations.  Here we consider some of the
parallels and differences between the two approaches.

Figure~3 of \citet{Richardson_primer} (a single world intervention
template, SWIT) is reproduced here as \figref{basicswig}, with
notation changed so as more closely to match our own.  Note the
splitting of the treatment node $T$.  As we shall see, this graph
encodes ignorability of the treatment assignment, and can thus be
compared with our own representations of ignorability.
\begin{figure}[htbp]
  \begin{center}
    \resizebox{2in}{!}{\includegraphics{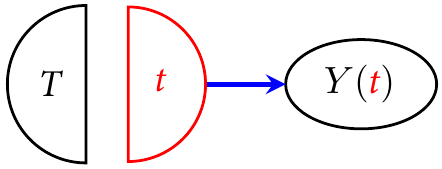}}
    \caption{Simple SWIG template, expressing PO (weak) ignorability}
    \label{fig:basicswig}
  \end{center}
\end{figure}

In \figref{basicswig}, $T$ denotes the treatment applied in the
observational regime: it thus corresponds to our
``intention-to-treat'' variable $T^*$.  The node labelled $t$
represents an intervention to set the treatment to $t$: it therefore
corresponds to $\check T=t$ in our development.  The variable $Y(t)$,
the ``potential response'' to the intervention at $t$, has no direct
analogue in our approach, but that is inessential, since only its
distribution is relevant; and that corresponds to our distribution
$P_t$ of $Y$ in response to the intervention $\check T=t$.

Applying the standard $d$-separation semantics to \figref{basicswig}
(ignoring the unconventional shapes of some of the nodes), the
disconnect between $T$ and $t$ represents their independence.  This
corresponds to our equation~\eqref{indst}, encapsulating the covariate
nature of $T^*$.  Further, by the lack of an arrow from $T$ to $Y(t)$,
the graph encodes $\indo {Y(t)} T$, which is to say that the
distribution of $Y(t)$---the outcome consequent on a (real or
hypothetised) intervention at $t$---is regarded as independent of the
intention-to-treat variable (and this property should hold for all
$t$).  In our notation, this becomes $\ind Y {T^*} {\check T}$, as
expressed in our equation \eqref{indys_t}, and represents ignorability
of the treatment assignment.  As described in \secref{ignore}, in our
treatment this can be represented by the DAG of
\figref{igndag}---which is therefore our translation of the SWIT of
\figref{basicswig}, conveying essentially the same information in a
different form.

Note that, in the approach of \citet{Richardson_primer}, in order
fully to capitalise on the ignorability preperty represented by
\figref{basicswig}, additional external use must be made of the
assumption of (functional) consistency: $T = t$ implies $Y(t) = Y$.
For example, in this approach the average causal effect \ace is
defined as $\E\{Y(0)-Y(1)\}$.  Now by ignorability, as represented in
the SWIT of \figref{basicswig}, $\indo {Y(t)} T$, whence
$\E\{Y(t)\} = \E\{Y(t)\mid T=t\}$.  But we then need to make further
use of functional consistency to replace this by $\E\{Y \mid T=t\}$,
so obtaining $\ace = \E\{Y \mid T=1\}-\E\{Y \mid T=0\}$.

Our analogue of functional consistency is distributional consistency
(\defref{distcons}):
\begin{math}
  Y \cd (T=t, F_T = \idle) \,\approx\, Y\cd (T^*=t, F_T = t).
\end{math}
However, this property has already been used in justifying the
representation by means of \figref{igndag}.  Once that graph is
constructed, distributional consistency does not require further
explicit attention since, as shown in \remref{ok}, it is already
represented in \figref{basic}, and thus in \figref{igndag}.  And then
\figref{collign} can be used directly to represent and manipulate the
fundamental DT representation of ignorability, as expressed by
\eqref{igndag}.  Thus we define
$\ace = \E(Y \mid F_T=1)-\E(Y \mid F_T=0)$.  With ignorability
expressed as $\ind Y {F_T} T$, as encoded in \figref{igndag}, we
immediately have
$\E(Y \mid F_T=t) = \E(Y \mid T=t, F_T=t) = \E(Y \mid T=t)$, and thus
$\ace = \E(Y \mid T=1)-\E(Y \mid T=0)$.

A further conceptual advantage of our approach is that is unnecessary
to consider (even one-at-a-time) the distinct potential
responses\footnote{unhelpfully described as ``counterfactuals'' by
  \citet{Richardson_primer}} $Y(t)$: we have a single response
variable $Y$, but with a distribution that may be regime-dependent.

\section{A comparative study:  $g$-computation}
\label{sec:gcomp}
In this Section we compare, contrast, and finally unify, the various
approaches to causal modelling and inference, in the context of the
specific example of \secref{dagex}.  We suppose we have observational
data, and wish to identify the distribution of $Y$ under interventions
at $X_0$ and $X_1$.  Purely for notational simplicity, we assume all
variables are discrete


\subsection{Pearl's {\em do\/}-calculus}

The {\em do\/}-calculus \cite[\S~3.4]{pearl:book} is a methodology for
discovering when and how, for a problem represented by a specified
Pearlian DAG, it is possible to use observational information to
identify an interventional distribution.  Notation such as
$p(x \mid y, \widehat z)$ refers to the distribution of $X$ given the
observation $Y=y$, when $Z$ is set by intervention to $z$.  Pearl
gives 3 rules, based on interrogation of the DAG, that allow
transformation of such expressions.  If by successive application of
these rules we can re-express our desired interventional target by a
hatless expression, we are done.

In this notation, we would like to identify
$p(y \mid \widehat x_0, \widehat x_1)$.  We can write 
\begin{equation}
  \label{eq:dodyn}                             
  p(y \mid \widehat x_0,\widehat x_1)  =\sum_z p(y \mid \widehat x_0,\widehat x_1,z)\times p(z \mid  \widehat x_0, \widehat x_1).                            
\end{equation}
 
According to Pearl's Rule~2, we have
\begin{equation}
  \label{eq:rule2}
  p(y \mid \widehat x_0,\widehat x_1,z) = p(y \mid x_0, x_1,z)  
\end{equation}
because $Y$ is $d$-separated from $(X_0,X_1)$ by $Z$ in the DAG of
\figref{pearlswiga} modified by deleting the arrows out of $X_0$ and
$X_1$.

Next, again by Rule~2, we can show
\begin{equation}
  \label{eq:rule2b}
  p(z \mid \widehat x_0, \widehat x_1) = p(z \mid x_0, \widehat x_1)  
\end{equation}
by seeing that $Z$ is $d$-separated from $X_0$ by $X_1$ in the DAG
modified by the deleting arrows into $X_1$ and out of $X_0$.

Finally, by Rule~3, we confirm
\begin{equation}
  \label{eq:rule3}
  p(z \mid x_0, \widehat x_1) = p(z \mid x_0)  
\end{equation}
because $Z$ is $d$-separated from $X_1$ by $X_0$ in the DAG with
arrows into $X_1$ removed.  So on combining \eqref{rule2b} and \eqref{rule3}
we have shown
\begin{equation}
  \label{eq:rules}
  p(z \mid \widehat x_0, \widehat x_1) = p(z \mid x_0).  
\end{equation}
Inserting \eqref{rule2} and \eqref{rules} into \eqref{dodyn}, we conclude
\begin{equation}
  \label{eq:gcomp}
  p(y \mid \widehat x_0,\widehat x_1)  = \sum_z p(y \mid  x_1, z) \times p(z \mid x_0),
\end{equation}
showing that the desired interventional distribution can be
constructed from ingredients identifiable in the observational regime.
Equation~\eqref{gcomp} is (a simple case of) the {\em
  $g$-computation\/} formula of \textcite{jr:mm}.

\subsection{DT approach}
\label{sec:dtgcomp}

As described in \textcite{apd:annrev}, the DT approach supplies a more
straightforward way of justifying and implementing {\em
  do\/}-calculus, using the augmented DAG.  In our problem this is
\figref{pearlstara}, and what we want is
$p(Y=y \mid F_0=x_0,F_1=x_1)$.

Noting $F_0=x_0 \Rightarrow X_0=x_0$ \etc, in general we have:
\begin{eqnarray}
  \nonumber
  p(Y=y \mid F_0=x_0,F_1=x_1)  &=& \sum_z p(Y=y \mid X_0=x_0,X_1=x_1,Z=z,F_0=x_0,F_1=x_1)\\
  \label{eq:gendyn}                              &&{}\times p(Z=z \mid  X_0=x_0, F_0=x_0,F_1=x_1).
\end{eqnarray}

Applying $d$-separation to \figref{pearlstara}, we can infer the
following conditional independencies:
\begin{eqnarray}
    \label{eq:rr6}
  Y &\cip& {(F_0,X_0,F_1)} \mid (Z, X_1)\\
  Z&\cip& (F_0 , F_1) \mid  X_0.
\end{eqnarray}
Using these in \eqref{gendyn} we obtain
\begin{eqnarray}
  \nonumber
  p(Y=y \mid F_0=x_0,F_1=x_1)  &=& \sum_z p(Y=y \mid  X_1=x_1, Z=z,F_0=\idle, F_1=\idle)\\
  \label{eq:gcompdt}
                               &&{}\times p(Z=z \mid X_0=x_0,F_0=\idle,F_1=\idle),
\end{eqnarray}
which is \eqref{gcomp}, re-expressed in DT notation.

\subsection{PO approach}
\label{sec:pogcomp}
The Pearlian/DT approach makes no use of potential outcomes.  By
contrast, these are fundamental to the original approach of Robins,
where the conditions supporting $g$-computation are:
  \begin{eqnarray}
    \label{eq:rr1}
    {Y(x_0,x_1)}&\cip&{X_1}\,\mid\,(Z, X_0=x_0)\\
  \label{eq:rr2}
  {Z(x_0)}&\cip& {X_0}.
\end{eqnarray}

In his Example~11.3.3, \textcite{pearl:book}, basing his argument on
his ``twin-network'' construction, claims that \eqref{rr1} can not be
derived from a PO interpretation of \figref{pearlswiga}.  However,
\citet{Richardson_primer} refute this by constructing the SWIT version
of \figref{pearlswiga}, as in \figref{pearlswigb}.
\begin{figure}[htbp]
   \begin{center}
  \includegraphics[width=.4\textwidth]{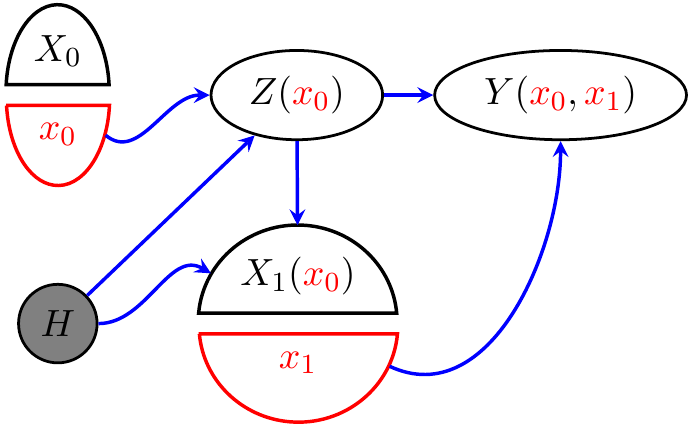}
  \caption{SWIT}
  \label{fig:pearlswigb}
\end{center}
\end{figure}

This DAG encodes the property
\begin{displaymath}
  {Y(x_0,x_1)}\,\cip\,{X_1(x_0)}\,\mid\,(Z(x_0), X_0)
\end{displaymath}
whence
\begin{equation}
  \label{eq:rr3}
  {Y(x_0,x_1)}\,\cip\,{X_1(x_0)}\,\mid\,(Z(x_0), X_0=x_0).
\end{equation}
They then apply functional consistency,
$X_0=x_0 \Rightarrow Z(x_0)=Z, X_1(x_0)= X_1$, to deduce \eqref{rr1}.
As for \eqref{rr2}, this is directly encoded in \figref{pearlswigb}.

\subsection{Unification}
We can use the DT approach to relate all the approaches above.

\subsubsection{DT for SWIG/PO}
\label{sec:dt4po}
\figref{pearlstarb}, using explicit ``intention to treat'' variables
and regime indicators, is the DT reinterpretation of the SWIT of
\figref{pearlswigb}.


From \figref{pearlstarb} (noting that the dotted arrow from $X_1^*$ to
$X_1$ disappears when $F_1 \neq \idle$) we can read off
\begin{displaymath}
  \ind Y {X_1^*} {Z, X_0, F_0, F_1=x_1},
\end{displaymath}
so that 
\begin{equation}
  \label{eq:rr4}
  \ind Y {X_1^*} {Z, X_0=x_0, F_0=\idle, F_1 = x_1},
\end{equation}
which is the DT paraphrase of \eqref{rr1}.  Similarly the DT
paraphrase of \eqref{rr2},
\begin{equation}
  \label{rr5}
  \ind Z {X_0^*} {F_0=x_0},
\end{equation}
is likewise encoded in \figref{pearlstarb}.  (In particular, both
these properties are consequences of our assumptions
\eqref{dagign0}--\eqref{dagign3}, together with \eqref{daginter}.)

\subsubsection{Consistency?}
\label{sec:cons}
Note that the derivations in \secref{dt4po} above do not require
further explicit application of (functional or distributional)
consistency conditions.  We could have complicated the analysis by
mimicking more closely that of \secref{pogcomp}.  The DT paraphrase of
\eqref{rr3}, which can be read off \figref{pearlstarb}, is
\begin{displaymath}
  \ind Y {X_1^*} {Z, X_0^*, F_0 = x_0, F_1 = x_1}.
\end{displaymath}
On restricting to $X_0^*=x_0$ and applying the distributional
consistency condition, we obtain the DT paraphrase of \eqref{rr1}:
\begin{displaymath}
  \ind Y {X_1^*} {Z, X_0=x_0, F_0 = \idle, F_1 = x_1}.
\end{displaymath}
But note that the required distributional consistency property can be
expressed as
\begin{displaymath}
  \ind Y {(X_1^*,F_0)} {(Z, X_0, F_1 = x_1)},
\end{displaymath}
and this is already directly encoded in \figref{pearlstarb}.  That
being the case, we can leave it implicit and shortcut the analysis, as
in \secref{dt4po}

\subsubsection{DT for Pearl}
\label{sec:dt4pearl}
We have shown that, if we can justify the DT ITT representation of
\figref{pearlstarb}, we can derive \eqref{rr1} and \eqref{rr2}, the
conditions used to derive the $g$-computation formula \eqref{gcomp} in
the PO approach.  However, the same end-point can be reached much more
directly.  Extracting from \figref{pearlstarb} the conditional
independencies between just the observable variables and the
intervention indicators (\ie, eliminating $X_0^*$ and $X_1^*$), we
recover \figref{pearlstara}, the DT version of the Pearlian DAG
\figref{pearlswiga}.  From this, as shown in \secref{dtgcomp},
\eqref{gcomp} can readily be deduced directly, without any need to
complicate the analysis by consideration of potential outcomes.  As
described in \secref{conver}, consideration of intention-to-treat
variables is needed to justify the appropriateness of the
augmented/Pearlian DAG of \figref{pearlstara}; but once that has been
done, for further analysis we can simply forget about the ITT
variables $X_0^*$ and $X_1^*$.

\textcite{dd:ss}, \S10.1.1, show how the PO conditions typically
imposed to justify more general forms of $g$-computation imply the
much simpler DT conditions supporting more straightforward
justification.  The DT approach can, moreover, be straightforwardly
extended to allow sequentially dependent randomised interventions,
which can introduce considerable additional complications for the PO
approach.

\section{Discussion}
\label{sec:disc}
In this paper we have developed a clear formalism for problems of
statistical causality, based on the idea that I want to use external
data to assist me in making a decision.  We have shown how this serves
as a firm theoretical foundation for methods framed within the DT
approach, enabling transfer of probabilistic information from an
observational to an interventional setting.  We have emphasised, in
particular, just what considerations are involved---and so what needs
to be argued for---when we invoke enabling assumptions such as
ignorability.  In the course of the development we have introduced DT
analogues of concepts arising in other causal frameworks, including
consistency and the stable unit-treatment value assumption, and
clarified the similarities and differences between the different
approaches.

General though our analysis has been, it could be generalised still
further.  For example, our exchangeability assumptions treat all
individuals on a par.  But we could consider more complex versions of
exchangeability, such as are relevant in experimental designs where we
distinguish various factors which may be crossed or nested
(\textcite{apd:symmods}, \textcite{apd:cinfer} \S~10.1); or conducted
more detailed modelling of non-exchangeable data.  Our analysis of
DAGs in this article has been restricted to non-randomised point
interventions, taking no account of information previously learned.
Further extension would be needed to fully justify, \eg, DT models for
dynamic regimes \cite{dd:ss}.


\newpage

\appendix

\section{Proof of \thmref{stardaggen}}
\label{sec:thmproof}

As in \remref{varind}, and purely as an instrumental tool, we regard
all the regime variables as stochastic and mutually independent:
\begin{equation}
  \label{eq:CIP}
  \CIP_{i=1}^k F_i.
\end{equation}
We shall show that $\dag^*$ then represents the conditional
independencies between all its variables.  The desired result will
then follow on conditioning on $F_{1:k}$.

For economy of notation, we write $W_i$ for $(V_i,X_i)$, $W_{a:b}$ for
$(V_{a:b},X_{a:b})$. and similarly $W_i^*$,  $W_{a:b}^*$.

\begin{lem}
  \label{lem:CIPP}
  For each $r=1,\ldots,k-1$,
  \begin{equation}
    \label{eq:needr}
    H_r:  \ind  {F_{r+1:k}} {F_{1:r}} {W_{1:r}}.
  \end{equation}
\end{lem}

\begin{proof}
  We  show \eqref{needr} by induction.

  By \eqref{CIP}, $\indo {F_{1}}{F_{2:k}}$, while by $A_2'$ we have
  $\ind {W_1} {F_{2:k}} {F_1}$.  Together these yield
  $\indo {(F_1, W_1)} {F_{2:k}}$, from which $H_1$ follows.
  
  Suppose now $H_r$ holds.  From $B_r'$ we have
  \begin{equation}
    \label{eq:ah}
    \ind {W_{r+1}} {F_{1:r}} {(F_{r+1:k},W_{1:r})}.
  \end{equation}
  Together with $H_r$ this gives
  \begin{equation}
    \label{eq:al}
    \ind {(F_{r+1:k}, W_{r+1})} {F_{1:r}} {W_{1:r}}
  \end{equation}
  whence 
  \begin{equation}
    \label{eq:am}
    \ind {F_{r+1}}  {F_{1:r}} {(F_{r+2:k},W_{1:r+1})}.
  \end{equation}

  Also, by $A_{r+2}'$,
  \begin{equation}
    \label{eq:fin}
    \ind{W_{1:r+1}} {F_{r+2:k}} {F_{1:r+1}},
  \end{equation}
  which together with $\indo{F_{1:r+1}} {F_{r+2:k}}$, from \eqref{CIP}, gives $\indo {(F_{1:r+1},W_{1:r+1})}  {F_{r+2:k}}$, from which we have
  \begin{equation}
    \label{eq:finn}
    \ind{F_{r+2:k}} {F_{1:r+1}}{W_{1:r+1}}.
  \end{equation}
  So $H_{r+1}$ holds and the induction is established.
\end{proof}

\begin{lem}
  \label{lem:vr}
  For each $r$:
  \begin{equation}
    \label{eq:eachr}
    \ind {(V_{r+1},X_{r+1}^*)} {(F_{1:k},X_{1:r}^*)} {(V_{1:r},X_{1:r})}.
  \end{equation}  
\end{lem}

\begin{proof}
 
  From $B_r'$, we have
  \begin{equation}
    \label{eq:irf}
  \ind {W_{r+1}^*} {F_{1:r}} {(F_{r+1:k},W_{1:r})}.
\end{equation}
Combining this with \eqref{needr} gives
\begin{displaymath}
  \ind  {(F_{r+1:k},W_{r+1}^*)} {F_{1:r}} {W_{1:r}},
\end{displaymath}
whence
\begin{equation}
  \label{eq:cc}
  \ind  {W_{r+1}^*} {F_{1:r}} {W_{1:r}}.
\end{equation}

Also, from $A_{r+1}'$,
  $$\ind {W_{r+1}^*} {F_{r+1:k}} {(F_{1:r},W_{1:r})}.$$
Together with \eqref{cc} this gives
  \begin{equation}
    \label{eq:z}
    \ind {W_{r+1}^*} {F_{1:k}} {W_{1:r}}.  
  \end{equation}  

 Also from $B_r'$ we have
 \begin{equation}
   \label{eq:chain}
   \ind {W_{r+1}^*} {X_{1:r}^*} {(F_{1:k},W_{1:r})}.
 \end{equation}
 Now combining \eqref{z} and \eqref{chain} we obtain \eqref{eachr}.
\end{proof}
\quad\\

To complete the proof of \thmref{stardaggen}, consider the sequence
  $$L^* = (F_1, \ldots, F_k, V_1,X_1^*,X_1,\ldots, V_k, X_k^*,X_k, V_{k+1}).$$
  which is consistent with the partial order of the ITT DAG $\dag^*$.
  Each $V_i$ may comprise a number of domain variables: we consider it
  as expanded into its constituent parts, respecting the partial order
  of $\dag$, and thus of $\dag^*$.

  To establish \thmref{stardaggen}, we show that each variable in
  $L^*$ is independent of its predecessors in $L^*$, conditional on
  its parent variables in $\dag^*$.

  \begin{enumerate}
  \item For each $F_i$, this holds by \eqref{CIP}.
  \item For an intervention target $X_i$, its only parents in $\dag^*$
    are $X_i^*$ and $F_i$.  By \eqref{daginter}, conditional on these
    $X_i$ is fully determined, hence independent of anything.
  \item \label{it:aa} Consider now a non-intervention domain variable,
    $U$ say.  Its parents in $\dag^*$ are the same as its parents in
    $\dag$.  Now $U$ is contained in $V_r$ for some $r$.  By
    \eqref{eachr} its conditional distribution, given all its
    predecessors in $L^*$, depends only on the preceding domain
    variables.  In particular, this conditional distribution, being
    the same in all regimes, must agree with that in the observational
    regime, whose independencies are encoded in the initial DAG
    $\dag$---and so depends only on the parents of $U$ in $\dag$, and
    hence in $\dag^*$.
  \item The remaining case, of an ITT variable $X_i^*$, follows
    similarly to \itref{aa}, on further noting that the parents of
    $X_i^*$ in $\dag^*$ are the same as the parents of $X_i$ in
    $\dag$, and $X_i^*$ is identical to $X_i$ in the observational
    setting.
  \end{enumerate}


\end{document}